\documentclass{fundam}

\usepackage[math]{cellspace}
\usepackage{graphicx}
\usepackage{pgf,tikz,pgfplots,pgfkeys}
\usepackage{amsmath,mathrsfs,amssymb}
\usepackage{bbm}
\usepackage{mathtools}
\usepackage[shortlabels]{enumitem}
\usepackage{colortbl}
\usepackage{multirow}
\usepackage{makecell}
\usepackage{tikz-3dplot} %r

\usepackage{frenchineq}
\usepackage[colorinlistoftodos,bordercolor=orange,backgroundcolor=orange!20,linecolor=orange,textsize=scriptsize]{todonotes}
\usetikzlibrary{calc}
\usetikzlibrary{shapes.geometric}
\usetikzlibrary{automata, positioning}
\usetikzlibrary{arrows}
\usetikzlibrary{arrows.meta}
\usepackage{xcolor}

\newcommand\phase[1]{{\footnotesize $\boxed{#1}$}}
\providecommand{\doi}[1]{\url{https://doi.org/{#1}}}

\newcounter{tikzfigures}

\newcommand{\scrS}{\mathscr{S}}
\newcommand{\semigroup}{\scrS}

\newcommand{\Pcal}{\mathcal{P}}
\newcommand{\Zcal}{\mathcal{Z}}
\newcommand{\ZQcal}{\mathcal{Z}^{\mathcal{Q}}}

\newcommand{\Fcal}{\mathcal{F}}
\newcommand{\Pprio}{\mathcal{P}_\textsf{prio}}

\newcommand{\Ppsel}{\mathcal{P}_\textsf{psel}}
\newcommand{\Qcal}{\mathcal{Q}}
\newcommand{\Qsync}{\mathcal{Q}_\textsf{sync}}
\newcommand{\Qprio}{\mathcal{Q}_\textsf{prio}}
\newcommand{\Qpsel}{\mathcal{Q}_\textsf{psel}}
\newcommand{\R}{\mathbb{R}}
\newcommand{\Rplus}{\mathbb{R}_{\geq 0}}

\newcommand{\N}{\mathbb{N}}
\newcommand{\Z}{\mathbb{Z}}
\newcommand{\dr}{\textsc{Dr.}}
\newcommand{\Germ}{\mathbb{G}}
\newcommand{\out}{^{\textrm{out}}}

\newcommand{\inc}{^{\textrm{in}}}

\newcommand{\ie}{i.e.}
\newcommand{\etc}{etc}

\newcommand{\cadlag}{{c\`adl\`ag}}

\newcommand{\reviewone}[1]{}
\newcommand{\reviewtwo}[1]{}
\newcommand{\reviewthree}[1]{}

\usepackage{wrapfig}
\usepackage[font=footnotesize]{caption}
\usepackage{ragged2e}
\usepackage{array}
\newcolumntype{L}[1]{>{\raggedright\let\newline\\\arraybackslash\hspace{0pt}}m{#1}}
\newcolumntype{C}[1]{>{\centering\let\newline\\\arraybackslash\hspace{0pt}}m{#1}}
\newcolumntype{R}[1]{>{\raggedleft\let\newline\\\arraybackslash\hspace{0pt}}m{#1}}

\usepackage{cleveref}
\newtheorem{correspondence}[definition]{Correspondence Theorem}
\newtheorem{paradox}[definition]{Paradox}

\usepackage{hyperref}

\begin{document}

\setcounter{page}{169}
\publyear{2021}
\papernumber{2086}
\volume{183}
\issue{3-4}

  \finalVersionForARXIV
 %%\finalVersionForIOS

\title{Piecewise Affine Dynamical Models of \\
 Petri Nets -- Application to  Emergency Call Centers\thanks{The second author is supported by a joint PhD grant of DGA and INRIA.
             The authors have been partially supported by the ``Investissement d'avenir'', r{\'e}f{\'e}rence ANR-11-LABX-0056-LMH,
               LabEx LMH and by a project of ``Centre des Hautes {\'E}tudes du Minist{\`e}re de l'Int{\'e}rieur'' (CHEMI).} }

\author{Xavier Allamigeon\thanks{Address for correspondence:   INRIA and CMAP, CNRS, \'Ecole polytechnique, IP Paris, France.  \newline \newline
          \vspace*{-6mm}{\scriptsize{Received  October 2020; \ revised June 2021.}}},  Marin Boyet,  St\'ephane Gaubert
          \\
   INRIA and CMAP, CNRS, \'Ecole polytechnique\\
    IP Paris, France\\
 Xavier Allamigeon, Marin Boyet,  Stephane Gaubert\\
\{xavier.allamigeon, marin.boyet,  stephane.gaubert\}@inria.fr
 }

\runninghead{X. Allamigeon et al.}{Piecewise Affine Dynamics of Timed Petri Nets}

\maketitle

\vspace*{-5mm}
\begin{abstract}
  We study timed Petri nets, with preselection and priority routing.
  We represent the behavior of these systems by piecewise affine dynamical systems.  We use tools from the theory of nonexpansive mappings to analyze these systems. We establish an equivalence theorem between priority-free fluid timed Petri nets
  and semi-Markov decision processes, from which we derive
  the convergence to a periodic regime and the polynomial-time computability of the throughput.
  More generally, we develop an approach inspired by tropical geometry,
 characterizing the congestion phases as the cells of a polyhedral complex.
  We illustrate these results by a current application to the performance evaluation
  of emergency call centers in the Paris area. We show that priorities
  can lead to a paradoxical behavior: in certain regimes,
  the throughput of the most prioritary task
  may not be an increasing function of the resources.
\end{abstract}
\begin{keywords}
Timed Petri net,
Performance evaluation,
Markov decision process,
Tropical geometry,
Emergency call center
\end{keywords}

\section{Introduction}
\paragraph{Motivation}

Emergency call centers exhibit complex synchronization and concurrency phenomena. Various types of calls induce diverse chains of actions, including reception of the call, instruction by experts, dispatch of emergency means and monitoring of operations in progress.
The processing of
calls is subject to priority rules, making sure that the requests
evaluated as the most urgent are treated first.
The present work originates from a specific case study, concerning
the performance evaluation of the medical
emergency call centers in Paris
and its inner suburbs, operated by four {\em Services d'aide m\'edicale urgente} (SAMU) of {\em Assistance Publique -- H\^opitaux de Paris} (AP-HP).
One needs
to evaluate performance indicators, like the throughput
(number of calls of different types that can be processed without
delay). One also needs to optimize the resources (e.g., personnel of different kinds) to guarantee a prescribed quality of service for a given
inflow of calls.

\definecolor{gray1}{gray}{0.82}
\definecolor{gray2}{gray}{0.94}
\setlength{\arrayrulewidth}{.8pt}

\begin{table}[h]
%\vspace{1mm}
\begin{center}
  \caption{Correspondence between Petri nets and semi-Markov~decision~processes}
  \label{table:correspondence} \vspace{-2mm}
{\footnotesize
\begin{tabular}{C{4.5cm}!{\color{white}{\vrule width .8pt}}C{6cm}C{0cm}}
  \rowcolor{gray1}  Timed Petri nets & Semi-Markov decision processes      & \\[.6ex] \arrayrulecolor{white}\hline
  \rowcolor{gray2}  Transitions      & States                              & \\[.6ex] \arrayrulecolor{white}\hline
  \rowcolor{gray2}  Places           & Actions                             & \\[.6ex] \arrayrulecolor{white}\hline
  \rowcolor{gray2}  Physical time    & Time remaining to live              & \\[.6ex] \arrayrulecolor{white}\hline
  \rowcolor{gray2}  Counter function & Finite horizon value function       & \\[.6ex] \arrayrulecolor{white}\hline
  \rowcolor{gray2}  Synchronization & Multiple actions           & \\[.6ex] \arrayrulecolor{white}\hline
  \rowcolor{gray2}  Preselection routing & Probabilistic moves        & \\[.6ex] \arrayrulecolor{white}\hline
  \rowcolor{gray2}  Priority routing & Negative probabilities              & \\[.6ex] \arrayrulecolor{white}\hline
  \rowcolor{gray2}  Throughput       & Average cost                        & \\[.6ex] \arrayrulecolor{white}\hline
  \rowcolor{gray2}  Bottleneck places  & Optimal policies                  & \\[.6ex] \arrayrulecolor{white}\hline
  \rowcolor{gray2}  Congestion phases & Cells of the average cost complex & \\[.6ex]
  \end{tabular}}\end{center}\vspace*{-5mm}
\end{table}

\paragraph{Contribution}
We develop a general method for the analysis
of the timed behavior of Petri nets, based
on a representation by piecewise linear dynamical systems.
These systems govern {\em counter functions}, which
yield the number of firings of transitions as a function of time.
We allow routings based either on preselection or priority rules.
Preselection applies to situations in which certain attributes
of a token determine the path it follows, e.g.~different
types of calls require more or less complex treatments.
Moreover, priority rules are used to allocate resources when conflicts arise.
We study a {\em fluid relaxation}
of the model, in which the numbers of firings can take real
values.
Supposing the absence of priority routing,
we establish a correspondence between timed Petri nets
and semi-Markov decision processes. \Cref{table:correspondence} provides the details of this correspondence that we shall discuss in the paper. Then, we apply methods from the theory
of semi-Markov decision processes to analyze timed
Petri nets.  We show that the counter variables converge
to a periodic orbit (modulo additive constants). Moreover, the
throughput can be computed in polynomial time, by
looking for affine stationary regimes and exploiting
linear programming formulations.
We also show that the throughput is given as a function of the resources (initial marking), by an explicit concave piecewise affine map. The cells
on which this map is affine yield a polyhedral complex,
representing the different ``congestion phases''. We finally
discuss the extension of these analytic results to the case
with priorities. The dynamics still has the form
of a semi-Markov type Bellman equation,
but with {\em negative} probabilities. Hence, the theoretical
tools used to show the convergence to a periodic orbit
do not apply anymore. However, we can look for
the affine stationary regimes, which turn out to be the
points of a tropical variety.
From this, we still obtain a phase diagram, representing
all the possible throughputs of stationary regimes.
Throughout the paper, these results are illustrated by the
case study of emergency call centers. The final section
focuses on the analysis of a policy
proposed by the SAMU, involving a monitored reservoir,
designed to handle without delay the most urgent calls.
We show that this particular model
has a paradoxical behavior in an exceptional congestion regime: increasing some resources
may result in a decrease of the throughput of the most prioritary task.

\paragraph{Related work}
Our approach originates from the max-plus modeling
of timed discrete event systems, introduced by Cohen, Quadrat and Viot
and further investigated by Baccelli and Olsder and a number of authors.
We refer the reader to the monographs~\cite{bcoq,how06}
and to the survey of Komenda, Lahaye,
Boimond and van den Boom~\cite{KOMENDA}.
The max-plus approach was originally developed for timed event graphs.
Cohen, Gaubert and Quadrat extended it to fluid Petri nets with preselection routing~\cite{CGQ95b,CGQ95a}.
Gaujal and Giua established in~\cite{gaujal2004optimal}
further results on the model of~\cite{CGQ95b,CGQ95a}.
Their results include a characterization of the throughput
as the optimal solution of linear program.
Recalde and Silva~\cite{recalde} obtained linear programming formulations
for a different fluid model.

By comparison with~\cite{CGQ95b,CGQ95a,gaujal2004optimal}, we use more
powerful results on semi-Markov decision processes and nonexpansive mappings. This allows us, in particular, to deduce
more precise asymptotic results, concerning the deviation $z(t)-\rho t$ between the counter function $z$ at time $t$ and its average growth $\rho t$,
instead of the mere existence of the limit $\lim_{t\to+\infty} z(t)/t=\rho$.
We also establish the existence of the latter limit even in the case of irrational holding times, and provide a polyhedral characterization
of this limit, in terms of the ``Throughput complex'' (Corollary~\ref{coro:complex}).
This characterization holds without any irreducibility assumption (an earlier
formula of this nature was stated in~\cite{CGQ95b} in the special
irreducible case). The present work is
a follow-up of~\cite{allamigeon2015performance}, in which B\oe{}uf and two of the authors
established an equivalence between timed Petri nets with priorities
and a class of piecewise-linear models.

The present methods are complementary to
  probabilistic approaches~\cite{lecuyer}.
  Priority rules put our systems outside the classes of exactly solvable probabilistic
  models; only scaling limit type results on suitably purified
  models are known~\cite{boeufrobert}.
  In contrast,
  fluid models allow one to compute phase portraits analytically. They lead to lower bounds of dimensioning which are accurate when the arrivals do not fluctuate, and which can subsequently be confronted with
  results of simulation.

\section{Piecewise affine models of timed Petri nets}
\subsection{Preliminaries on timed Petri nets}

A {\em timed Petri net} is given by a
bipartite graph whose vertices are either {\em places} or {\em transitions}. We denote by $\Pcal$ (resp.\ $\Qcal$) the finite set of places (resp.\ transitions). For two vertices $x$ and $y$ forming a place-transition pair, $x$ is said to be an upstream (resp.\ downstream) vertex of $y$ if there is an arc of the graph going from $x$ to $y$ (resp. from $y$ to $x$). The set of upstream (resp.\ downstream) vertices of $x$ is denoted by $x\inc$ (resp.\ $x\out$).

Every place $p$ is equipped with an {\em initial marking} $m_p \in \N$,  representing the number of tokens initially present in the place before starting the execution of the Petri net.
The place $p$ is also equipped with a holding time $\tau_p \in \Rplus$, so that a token entering $p$ must sojourn in this place at least for a time $\tau_p$ before becoming available for firing a downstream transition. In contrast, firing a transition is instantaneous. Every arc from a place $p$ to a transition $q$ (resp.\ from a transition $q$ to a place $p$) is equipped with a positive and integer weight denoted by $\alpha_{qp}$ (resp.~$\alpha_{pq}$). Transition $q$ can be fired only if each upstream place $p$ contains $\alpha_{qp}$ tokens. In this case, one firing of the transition $q$ consumes $\alpha_{qp}$ tokens in each upstream place $p$, and creates $\alpha_{p'q}$ in each downstream place~$p'$. Unless specified, the weights are set to $1$.
The same transition can be fired as many times as necessary, as long as tokens in the upstream places are available. We shall assume that {\em transitions are fired as soon as possible}. By convention, the tokens of the initial marking are all available when the execution starts.

When a place has several downstream transitions, we must provide a {\em routing rule} specifying which transition is to be fired once a token is available. We distinguish two sets of rules: priority and preselection.

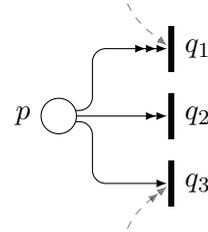
\begin{wrapfigure}{R}{0.3\textwidth}
  \vspace{-.7cm}
  \centering
   % ./petri_pattern_prio.tex

\begin{tikzpicture}[scale=0.37]
\tikzset{place/.style={draw,circle,inner sep=4.7pt, fill=white}}
\tikzset{transition/.style={rectangle, thick,fill=black, minimum width=6mm,inner ysep=1pt}}
\tikzset{transitionV/.style={rectangle, thick,fill=black, minimum height=6mm,inner xsep=1pt}}
\tikzset{jeton/.style={draw,circle,fill=red,inner sep = .5pt}}
\tikzset{arrowPetri/.style={>=latex,rounded corners=5pt}}

\begin{scope}[shift={(0,0)},rotate=90]
    \coordinate (p) at (0,4) {};

    \node[transitionV, label = right:{$q_3$}] (q1) at (-2.4,0) {};
    \node[transitionV, label = right:{$q_2$}] (q2) at (0,0) {};
    \node[transitionV, label = right:{$q_1$}] (q3) at (2.4,0) {};

    \draw[->,arrowPetri] ($(p)+(-.2,0)$) -- ($(p)+(-0.2,-1.2)$) -| (q1);
    \draw[->>,arrowPetri] (p) -- (q2);
    \draw[->>>,arrowPetri] ($(p)+(.2,0)$) -- ($(p)+(0.2,-1.2)$) -| (q3);

    \node[place, label = left:{$p$}] (pf) at (p) {};

    \draw[->>,arrowPetri, dashed, opacity=.5] ($(q1)+(-1.6, 1.55)$) -- ($(q1)+(-.95, 1.2)$) -- (q1)  {};
    \draw[->,arrowPetri, dashed, opacity=.5] ($(q3)+(1.6, 1.55)$) -- ($(q3)+(.95, 1.2)$) -- (q3)  {};

\end{scope}
\end{tikzpicture}

  \caption{Priority routing}\label{fig:priority_routing}
  \vspace{-.7cm}
\end{wrapfigure}

A {\em priority routing} on a place $p$ is specified by a total order $\prec_p$ over the downstream transitions of $p$. The principle of this routing rule is that a transition $q \in p\out$ is fired only if there is no other fireable transition $q' \in p\out$ with a higher priority, \ie~$q' \prec_p q$ (or equivalently $q \succ_p q'$).
We represent the ordering of downstream transitions by a variable number of arrow tips, like in~\Cref{fig:priority_routing}, with the convention that the highest priority transition (the minimal element of $p\out$ with respect to $\prec_p$) is the one pointed by the highest number of tips.

Priority routing will be used in our model of monitored reservoir studied in~\Cref{sec:casestudyprio}. We denote by $\Qprio$ the subset of $\Qcal$ consisting of the downstream transitions of places subject to priority routing. We allow transitions in $\Qprio$ to admit multiple upstream places ruled by priority routings as long as the following compatibility condition is met.
\begin{definition}\label{def:compatibility_condition}
Let $\Pprio$ denote the set of places subject to priority routing. We say that the rules $(\prec_p)_{p \in \Pprio}$ are {\em compatible} if their union (as binary relations) is acyclic.
\end{definition}

Acyclicity means that the transitive closure of the union of the local total orders $(\prec_p)_{p \in \Pprio}$ forms a global partial order on the set $\Qcal$ of all transitions.

\medskip
The {\em preselection routing} on a place $p$ is described by a collection of nondecreasing maps $(\Pi^p_q)_{q\in p\out}$ from $m_p+\N$ to $\N$ satisfying the property:
\begin{equation*}
\forall n\in \N \quad\textrm{s.t.}\quad n\geq m_p\, , \; \sum_{q\in p\out} \Pi^p_q(n) = n \, .
\end{equation*}
For $q\in p\out$, $\Pi^p_q(n)$ represents the number of tokens which are reserved to fire transition $q$, amongst the $n$ first tokens to enter place $p$ (including the initial marking $m_p$). In other words, they cannot be used to fire any other transition of $p\out$.
A natural example of preselection routing is the {\em proportional periodic routing}: if $p\out=\{q_1,q_2,\dots,q_k\}$, consider a positive integer $L$, a partition $(J_1,J_2,\dots,J_k)$ of $\{1,2,\dots,L\}$ and define $\Pi^p_{q_k}(n) = \mathrm{card}(\{1,2,\dots,n\}\cap (J_k+L\N))$. For large values of $n$, we have $\Pi^p_{q_k}(n)\sim n\cdot\mathrm{card}(J_k)/L$.

\medskip
In order to simplify the presentation of our following dynamical model, we assume that preselection routing is only allowed for places whose downstream transitions do not admit other upstream places.
The firing rule of the general case may be defined by reduction to this one by introducing extra places with holding time $0$, as illustrated on \Cref{fig:preselection}.

\begin{figure}[h]
\centering
\begin{minipage}[b]{0.65\textwidth}
 % ./petri_pattern_preselection.tex

\begin{tikzpicture}[scale=0.37]
\tikzset{place/.style={draw,circle,inner sep=4.7pt, fill=white}}
\tikzset{transition/.style={rectangle, thick,fill=black, minimum width=6mm,inner ysep=1pt}}
\tikzset{transitionV/.style={rectangle, thick,fill=black, minimum height=6mm,inner xsep=1pt}}
\tikzset{jeton/.style={draw,circle,fill=red,inner sep = .5pt}}
\tikzset{arrowPetri/.style={>=latex,rounded corners=5pt}}

\begin{scope}[shift={(0,0)}, rotate=0]
    \coordinate (p) at (0,0) {};

    \node[transitionV, label = right:{$q_1$}] (q1) at ( 3, 1.5) {};
    \node[transitionV, label = right:{$q_2$}] (q2) at ( 3,-1.5) {};

    \draw[->,arrowPetri] ($(p)+(0,.2)$) -- ($(p)+(1.2,.2)$)  |- (q1);
    \draw[->,arrowPetri] ($(p)+(0,-.2)$) --  ($(p)+(1.2,-.2 )$) |- (q2);

    \node[place,label=left:{$p$}] (pf) at (0,0) {};

    \draw[->,arrowPetri, dashed, opacity=.5] ($(q1)+(-1.55, 1.6)$) -- ($(q1)+(-1.2, .95)$) -- (q1)  {};
    \draw[->,arrowPetri, dashed, opacity=.5] ($(q2)+(-1.55,-1.6)$) -- ($(q2)+(-1.2,-.95)$) -- (q2)  {};

\end{scope}

\node (lbl) at (5,0) {$:=$};

\begin{scope}[shift={(8,0)}, rotate=0]
    \coordinate (p) at (0,0) {};

    \node[place, label = above:{$p_1^*$}] (p1) at (6,1.5) {};
    \node[place, label = below:{$p_2^*$}] (p2) at (6,-1.5) {};

    \node[transitionV, label = above right:{$q_1^*$}] (q1) at ( 3, 1.5) {};
    \node[transitionV, label = below right:{$q_2^*$}] (q2) at ( 3,-1.5) {};
    \node[transitionV, label = right:{$q_1$}] (q1b) at ( 9, 1.5) {};
    \node[transitionV, label = right:{$q_2$}] (q2b) at ( 9,-1.5) {};

    \draw[->,arrowPetri] ($(p)+(0,.2)$) -- ($(p)+(1.2,.2)$)  |- (q1);
    \draw[->,arrowPetri] ($(p)+(0,-.2)$) --  ($(p)+(1.2,-.2 )$) |- (q2);

    \node[place,label=left:{$p$}] (pf) at (0,0) {};

    \draw[->,arrowPetri] (q1) -- (p1);
    \draw[->,arrowPetri] (q2) -- (p2);
    \draw[->,arrowPetri] (p1) -- (q1b);
    \draw[->,arrowPetri] (p2) -- (q2b);

    \draw[->,arrowPetri, dashed, opacity=.5] ($(q1b)+(-1.55, 1.6)$) -- ($(q1b)+(-1.2, .95)$) -- (q1b)  {};
    \draw[->,arrowPetri, dashed, opacity=.5] ($(q2b)+(-1.55,-1.6)$) -- ($(q2b)+(-1.2,-.95)$) -- (q2b)  {};
\end{scope}
\end{tikzpicture}\vspace*{-3mm}

\captionsetup{justification=centering}
\caption{Compact notation for preselection\\ routing in case of multiple upstream places}
\label{fig:preselection}
\end{minipage}
\hfill
\begin{minipage}[b]{0.33\textwidth}
\centering
 % ./petri_pattern_sync.tex

\begin{tikzpicture}[scale=0.37]
\tikzset{place/.style={draw,circle,inner sep=4.7pt, fill=white}}
\tikzset{transition/.style={rectangle, thick,fill=black, minimum width=6mm,inner ysep=1pt}}
\tikzset{transitionV/.style={rectangle, thick,fill=black, minimum height=6mm,inner xsep=1pt}}
\tikzset{jeton/.style={draw,circle,fill=red,inner sep = .5pt}}
\tikzset{arrowPetri/.style={>=latex,rounded corners=5pt}}

\begin{scope}[shift={(0,0)}, rotate=0]
    \coordinate (p)  at (0,0) {};
    \coordinate (p2) at (0,3) {};
    \coordinate (p3) at (0,1.5) {};

    \node[opacity = 0] at (0,-1.8) {\;};

    \node[transitionV, label = right:{$q$}] (q1) at ( 3, 1.5) {};

    \draw[->,arrowPetri] ($(p)+(0,0)$) --    ($(p)+(2,0)$)    |-  ($(q1)+(-.1,-.3)$);
    \draw[->,arrowPetri] ($(p2)+(0,0)$) --  ($(p2)+(2,0 )$) |- ($(q1)+(-.1,.3)$);

    \node[place,label=left:{$p_1$}] (pf)  at (p) {};
    \node[place,label=left:{$p_2$}] (pf2) at (p2) {};

\end{scope}
\end{tikzpicture}\vspace*{-3mm}

\captionsetup{justification=centering}
\caption{A synchronization pattern\\ \textcolor{white}{\_}}
\label{fig:sync}
\end{minipage}\vspace{-1.5mm}
\end{figure}

We denote by $\Qpsel$ the subset of $\Qcal$ consisting of the downstream transitions of places ruled by preselection routing. By construction, we have $\Qpsel\cap\Qprio = \varnothing$. We define $\Qsync \coloneqq \Qcal\setminus(\Qpsel\cup\Qprio)$, \ie~the set of transitions with no upstream place ruled by preselection or priority routing. As a result, we have a partition of $\Qcal$ into $\Qprio$, $\Qpsel$, and $\Qsync$. Transitions of $\Qsync$ correspond to a synchronization pattern between several upstream places, as illustrated in~\Cref{fig:sync}. We point out that transitions with one upstream place can be of any of the three kinds $\Qprio$, $\Qpsel$, and $\Qsync$. The choice of their classification does not affect the analysis developed below.
\begin{remark}
Contrary to the preselection routing, the priority routing is essentially non-monotone (and therefore shall be left aside in Section~\ref{sec:correspondence}). Indeed, a ``fresh'' token might activate some prioritized transition before some other ``older'' token activates a non-prioritized transition.
\end{remark}

\subsection{Dynamic equations governing counter functions}

We associate with every transition $q\in \Qcal$ a {\em counter function} $z_q$  from $\R$ to $\Rplus$ such that $z_q(t)$ represents the number of firings of transition $q$ that occurred up to time $t$ included. Similarly, given a place $p\in \Pcal$, we denote by $x_p(t)$ the number of tokens that have entered place $p$ up to time $t$ included, taking into account the tokens initially present in $p$. By construction, $x_p$ and $z_q$ are non-decreasing {\cadlag} (right continuous with left limits) functions. Given a {\cadlag} function $f$, we denote by $f(t^-)$ the left limit at the point $t$. It may be smaller than $f(t)$.

\medskip
For each place $p\in\Pcal$, $x_p(t)$ is
given by the sum of the initial marking $m_p$ and the number of firings of transitions $q\in p\inc$ weighted by $\alpha_{pq}$ (recall that one firing of transition $q$ outputs $\alpha_{pq}$ tokens in $p$):
\begin{equation}
  \forall p\in\Pcal, \quad x_p(t) = m_p + \sum_{q\in p\inc} \alpha_{pq} \, z_q(t)
  \, .
  \label{eq:petri1}
  \tag{P1}
  \end{equation}

For each transition $q\in\Qcal$, the equation satisfied by $z_q$ depends on the routing policy of its upstream places. Suppose $q\in\Qsync$, so that its upstream places only admit $q$ for downstream transition. Since transitions are fired as early as possible and must wait for all upstream tokens to be available, we have:
\begin{equation}
z_q(t) = \min_{p\in q\inc}\; \bigg\lfloor \alpha_{qp}^{-1} x_p(t-\tau_p)\bigg\rfloor \, .
\label{eq:petri2}
\tag{P2}
\end{equation}
where $\lfloor\cdot\rfloor$ denotes the floor function (recall that $z_q$ must be integer). Suppose now that $q\in\Qpsel$. Because $q$ admits only one upstream place $p$, we also have:
\begin{equation}
   \quad z_q(t) = \bigg\lfloor \alpha_{qp}^{-1} \Pi^p_q(x_p(t-\tau_p))\bigg\rfloor \, .
  \label{eq:petri3}
  \tag{P3}
  \end{equation}
Finally, suppose that $q\in\Qprio$. We have
\begin{equation}
z_q(t) = \min_{p\in q\inc} \bigg\lfloor \alpha_{qp}^{-1}\bigg( x_p(t-\tau_p) - \sum_{q' \mathrel{\prec_p}\, q}\alpha_{q'p}z_{q'}(t) - \sum_{q' \mathrel{\succ_p}\, q}\alpha_{q'p}z_{q'}(t^-)   \bigg) \bigg\rfloor \, .
\label{eq:petri4}
\tag{P4}
\end{equation}

This equation can be interpreted by examining $z_q(t)-z_q(t^-)$, which represents the number of firings of $q$ at time $t$. The amount of tokens available in place $p\in q\inc$ at time~$t^-$ is $x_p(t-\tau_p)-\sum_{q'\in p\out}\alpha_{q'p}z_{q'}(t^-)$. However, transitions with higher priority than $q$ relatively to $p$ fire \linebreak $\sum_{q'\prec_p\, q}\alpha_{q'p}(z_{q'}(t)-z_{q'}(t^-))$ of these tokens, leaving $x_p(t-\tau_p) - \sum_{q' \mathrel{\prec_p}\, q}\alpha_{q'p}z_{q'}(t) -$  \linebreak $\sum_{q' \mathrel{\succeq_p}\, q}\alpha_{q'p}z_{q'}(t^-)$ available to fire $q$.
Equation~\eqref{eq:petri4} is obtained by packing these tokens in an integer number of groups of $\alpha_{qp}$ and taking the minimum of such terms over $q\inc$.

\definecolor{gray1}{gray}{0.82}
\definecolor{gray2}{gray}{0.94}
\setlength{\arrayrulewidth}{.8pt}
\begin{table}[b]
\vspace*{-2mm}
 \caption{Dynamic equations followed by transitions counter functions}  \label{table:equation_counters}
  \vspace{-2mm}
 \makebox[\textwidth][c]{
    \begin{tabular}{C{1.6cm}!{\color{white}{\vrule width .8pt}}S{L{13cm}}}
      \rowcolor{gray1} Type  & \multicolumn{1}{c}{Counter equation in the continuous model} \\
      \arrayrulecolor{white}\hline
      \rowcolor{gray2} \(q\in\Qsync\)  &  \(\displaystyle \, z_q(t)= \min_{p\in q\inc} \alpha_{qp}^{-1}\bigg(m_p + \!\! \sum_{q'\in p\inc} \alpha_{pq'} \, z_{q'}(t-\tau_p)\bigg)\)\\
      \arrayrulecolor{white}\hline
      \rowcolor{gray2} \(q\in\Qpsel\) &  \(\displaystyle \, z_q(t)=  \pi_{qp}\cdot\alpha_{qp}^{-1}\bigg(m_p + \!\! \sum_{q'\in p\inc} \alpha_{pq'} \, z_{q'}(t-\tau_p)\bigg) \)\\
      \arrayrulecolor{white}\hline
      \rowcolor{gray2} \(q\in\Qprio\)  & \(\displaystyle \, z_q(t)= \min_{p\in q\inc} \alpha_{qp}^{-1}\bigg(m_p + \!\! \sum_{q'\in p\inc} \alpha_{pq'} \, z_{q'}(t-\tau_p)-\!\!\sum_{q' \mathrel{\prec_p} q}\!\alpha_{q'p}z_{q'}(t)-\!\!\sum_{q' \mathrel{\succ_p} q}\!\alpha_{q'p}z_{q'}(t^-)\bigg)\)\\ \arrayrulecolor{white}
      \hline
    \end{tabular}   }
 \end{table}

The correspondence between the semantics of timed Petri net (expressed in terms of a transition system acting over states corresponding to timed markings) and the equations above has been proved in~\cite{allamigeon2015performance} in a more restricted model. It carries over to the current setting, allowing multiple levels of priority, preselection routings, and arcs with valuations.

It will be convenient
to consider the continuous relaxation of the previous dynamics. This boils down to considering infinitely divisible tokens and real-valued counters functions. The weights $\alpha_{pq}$ or $\alpha_{qp}$ can now be allowed to take positive real values. The priority and preselection routing rules are not affected by the fluid approximation, though in what follows we choose to focus only on proportional preselection routing: if a place $p$ is ruled by preselection, we fix a
stochastic vector
$(\pi_{qp})_{q\in p\out}$ such that $\Pi^p_q(x) = \pi_{qp} x$ for $x\geq m_p$.
Equivalently, this corresponds to the continuous relaxation of a stochastic routing at place $p$, in which $\pi_{qp}$ is the probability for a token to be routed to transition $q$. Finally, the continuous relaxation drops the floor functions. This leads to the dynamical system presented in~\Cref{table:equation_counters}, governing the counter functions $z_q$ of the transitions.

\section{Models of medical Emergency Call Centers}

We next present two models based on an ongoing collaboration with the Emergency Medical Services (EMS) of Paris and its inner suburbs (SAMU 75, 92, 93, and 94 of AP-HP).
In France, the nation-wide phone number 15 is dedicated to medical distress calls, dispatched to regional call centers. The calls are first answered by an operator referred to as a {\em medical regulation assistant} (MRA),
who categorizes the request, takes note of essential personal information and transfers the call to one of the following two types of physicians, depending on the estimated severity of the case:
\begin{enumerate}[(i)]
\item\label{item:case1} an emergency doctor, able to dispatch Mobile Intensive Care Units or first-responding ambulances and to swiftly send the patient to the most appropriate hospital unit;
\item\label{item:case2} a general practitioner, who can dispatch ambulances and provide medical advice.
\end{enumerate}
An MRA may also handle the call without transferring it if a conversation with a physician is not needed (report from a medical partner, dial error, \etc.).
\reviewtwo{
  There are some specificities of the model that call for precision/explanations. First, the probability to be a very urgent case is pi, but the model of TPNs with priority routing does not allow probabilities. One has to figure that probabilities are modeled with the alpha places, but the value for alpha is supposed to be an integer from the definitions of section 2. So one needs more explanations on the behavior of place 1 and transitions 2 \& 3 to be convinced that this is indeed a TPN-PR. The shortcut of Figure 4 is particularly misleading, as it lets the reader believe that place contents in place 1 is distributed to transitions, as in fact a certain amount of tokens is transferred to intermediate places that are then used by transitions. This leaves some ambiguity in the semantics: suppose that a place has one token, and two multiplying places connected to it, respectively with values x,y. Does it means that x+y tokens are produced out of one token (and then consumed by their respective transitions )? Or that either x or y tokens are produced ? The first case is an MDP behavior, while the second is more adapted to Petri nets.
}

In case~\ref{item:case1},
the MRA must wait for an emergency physician to be available before transferring the call, in order to report the details of the request. In this way, the patient is constantly kept on line with an interlocutor.
In case~\ref{item:case2}, patients are left on hold, and dealt with by general practitioners who answer the calls in the order of arrival.
As a first step, our main focus is the coupling between the answering operator and the emergency physician (which is a critical link of the system). Thus, for the sake of simplicity, we do not take into account what happens to calls in case~\ref{item:case2} after the MRA is released. In other words, we consider a simplified model in which only two types of inbound calls can occur: the ones which require the MRA to wait for an emergency physician and the ones which do not. We shall also consider that the patients do not leave the system before their call is picked up (infinite patience assumption).

\begin{figure}[h]
  \centering
  \def\tkzscl{.29}
  \vspace{-.2cm}
   % emergency_PN_usecase1.tex

\definecolor{colorARM}{rgb}{0,0,1}
\definecolor{colorARMres}{rgb}{.92,.5,.11}
\definecolor{colorAMU}{rgb}{1,0,0}
\definecolor{colorexit}{rgb}{0.4,0.4,0.4}

\tikzset{place/.style={draw,circle,inner sep=2.5pt,semithick}}
\tikzset{transition/.style={rectangle, thick,fill=black, minimum width=2mm,inner ysep=0.5pt, minimum width=2mm}}
\tikzset{jeton/.style={draw,circle,fill=black!80,inner sep=.35pt}}
\tikzset{pre/.style={=stealth'}}
\tikzset{post/.style={->,shorten >=1pt,>=stealth'}}
\tikzset{-|/.style={to path={-| (\tikztotarget)}}, |-/.style={to path={|- (\tikztotarget)}}}
\tikzset{bicolor/.style 2 args={dashed,dash pattern=on 1pt off 1pt,#1, postaction={draw,dashed,dash pattern=on 1pt off 1pt,#2,dash phase=1pt}}}
\tikzset{arrowPetri/.style={>=latex,rounded corners=5pt,semithick}}

\begin{tikzpicture}[scale=\tkzscl,font=\scriptsize]

\def\p{2.2}

\begin{scope}[shift={(0,4)}]
\node[place] (pool_arm) at ($(-2*\p,-1.5*\p)$) {};
\node (txt_Na) at ($(pool_arm)+(-1.3,0)$) {$N_A$};

\node[transition]    (q_arrivals)                          at    (0, 0) {};
\node[place]         (p_inc_calls)                         at    ($(q_arrivals) + (0, \p)$)  {};
\node[transition]    (q_inc_calls)                         at    ($(p_inc_calls) + (0, \p)$)  {};
\node[place]         (p_arrivals)        at    ($(q_arrivals) + (0,-\p)$) {};
\node[transition]    (q_debut_NFU)       at    ($(p_arrivals) + (0,-\p)$) {};
\node[transition]    (q_debut_AMU)         at    ($(p_arrivals) + (2*\p,-\p)$) {};

\draw[->,arrowPetri,colorARM] (p_arrivals)  |- ($(p_arrivals)+(.5*\p,-.5*\p)$) -| (q_debut_AMU);
\draw[->,arrowPetri,colorARM] (q_debut_NFU) |- ($(q_debut_NFU)+(0,-.5*\p)$) -| (pool_arm);
\draw[->,arrowPetri,colorARM] (p_arrivals)  -- (q_debut_NFU);
\draw[->,arrowPetri,colorARM] (q_arrivals) -- (p_arrivals);

\node[place]         (p_synchro)           at    ($(q_debut_AMU)  + (0,-\p)$) {};
\node[transition]    (q_unsynchro)         at    ($(p_synchro)  + (0,-\p)$) {};

\draw[arrowPetri]    (q_inc_calls) |- ($(q_inc_calls)+(-.35*\p,.25*\p)$);
\draw[dashed, arrowPetri]    (q_inc_calls) |- ($(q_inc_calls)+(-1.5*\p,.25*\p)$);
\draw[->,arrowPetri] (q_inc_calls) -- (p_inc_calls);
\draw[->,arrowPetri]                                    (p_inc_calls) -- (q_arrivals);
\draw[->,arrowPetri,colorARM]                           (q_unsynchro) -- ($(q_unsynchro)+(0,-.5*\p)$)  -|  (pool_arm);

\node[place]    (p_consult_AMU)    at    ($(q_unsynchro)+(0,-\p)$) {};

\draw[->,arrowPetri,bicolor={colorARM}{colorAMU}]    (q_debut_AMU) -- (p_synchro);
\draw[->,arrowPetri,bicolor={colorARM}{colorAMU}]    (p_synchro) -- (q_unsynchro);

\draw[->,arrowPetri,colorARM]                           (pool_arm)    |- ($(q_arrivals)+(-.5*\p,.5*\p)$) -- (q_arrivals);

\node (txt_taus) at ($(p_synchro)+(.5*\p,0)$) {$\tau_2$};
\node (txt_tau1) at ($(p_arrivals.center)+(.5*\p,0)$) {${\tau}_1$};
\node (txt_tau3) at ($(p_consult_AMU.center)+(.5*\p,0)$) {${\tau}_3$};
\node (txt_pi)  at ($(p_arrivals.center)+(-1.4,-1.65)$) {$1\!-\!\pi$};
\node (txt_pi3) at ($(p_arrivals.center)+(3.7, -1.65)$){$\pi$};

\node[place]      (pool_amu_old)    at    ($(q_unsynchro)+(2*\p,.5*\p)$) {};
\node (txt_Nm) at ($(pool_amu_old)+(-1.3,0)$) {$N_P$};

\node[transition] (q_end_consult_AMU)                  at    ($(p_consult_AMU)+(0,-\p)$) {};
\draw[->,arrowPetri,colorAMU]    (q_unsynchro)       -- (p_consult_AMU);
\draw[->,arrowPetri,colorAMU]    (p_consult_AMU)     -- (q_end_consult_AMU);
\draw[->,arrowPetri,colorAMU]    (q_end_consult_AMU) -- ($(q_end_consult_AMU)+(0,-1)$) -| (pool_amu_old);
\draw[->,arrowPetri,colorAMU]    (pool_amu_old)      |- ($(q_debut_AMU)+(.5*\p,.5*\p)$)      -- (q_debut_AMU);

\node (txt_z0) at ($(q_inc_calls)+(1*\p,0)$) {$z_0 = \lambda t$};
\node (txt_z1) at ($(q_arrivals)+(.5*\p,0)$) {$z_1$};
\node (txt_z2) at ($(q_debut_NFU)+(.5*\p,0)$) {$z_2$};
\node (txt_z3) at ($(q_debut_AMU)+(.5*\p,0)$) {$z_3$};
\node (txt_z4) at ($(q_unsynchro)+(.5*\p,0)$) {$z_4$};
\node (txt_z5) at ($(q_end_consult_AMU)+(.5*\p,0)$) {$z_5$};
\end{scope}
\end{tikzpicture}	
  \caption{A basic model of emergency call center~\eqref{SAMU1}}
  \vspace{-.6cm}
  \end{figure}
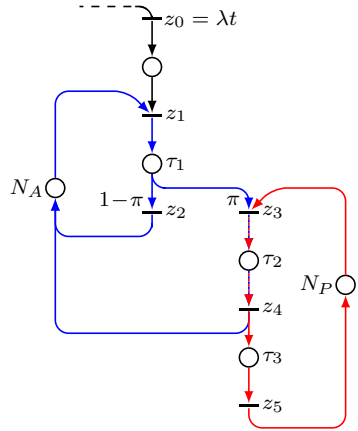

We represent this emergency call center by the (EMS-A) Petri net.
Inbound calls arrive \textit{via} the uppermost transition $z_0$. We may assume in what follows that $z_0(t)=\lambda t$ (arrivals at constant rate $\lambda$). The pool of MRAs is represented by the place with initial marking $N_A$. Transition $z_1$ is fired as soon as an MRA is available and a call is waiting for pick-up. Preliminary examination and information filling occur in place with holding time $\tau_1$, ruled by preselection routing: a known fraction $\pi$ of the patients are deemed to need the help of the emergency physician; for the complementary fraction $1 - \pi$ of the patients, the MRA is released at the firing of $z_2$. Transition $z_3$ is fired once a doctor is available from the pool of emergency physicians with initial marking $N_P$ and an MRA waits for transfer. Summarizing the case takes a time $\tau_2$ for both agents, then the firing of $z_4$ releases the MRA and the physician proceeds to the medical consultation with the patient for a time $\tau_3$ before getting released by the firing of $z_5$. We use the color blue (resp.~red) to highlight the circuits involving the MRA (resp.~the emergency doctor). For the sake of readability, patient exits at transitions $z_2$ and $z_5$ are not depicted.

\medskip
Applying the equations of the continuous relaxation of a timed Petri net recorded in~\Cref{table:equation_counters}, we obtain the following system of equations for the counter functions associated with transitions, where $x \wedge y$ stands for $\min(x,y)$.

\begin{equation}\left\{
  \begin{array}{rcrcl}
  z_1(t) & = & z_0(t)\;\;\;\;\; & \wedge & \big(N_{A} + z_2(t) + z_4(t)\big)\\[1ex]
  z_2(t) & = & (1-\pi)z_1(t-\tau_1) & &\\[1ex]
  z_3(t) & = & \pi  z_1(t-\tau_1) &\wedge& \big(N_{P} + z_5(t)\big)\\[1ex]
  z_4(t) & = & z_3(t-\tau_2) & &\\[1ex]
  z_5(t) & = & z_4(t-\tau_3)& &
  \end{array}\right.
  \label{SAMU1}
  \tag{EMS-A}
\end{equation}

As we shall see in~\Cref{sec:correspondence}, a slowdown arising either in the MRAs circuit or in the physicians circuit causes a slowdown of the whole system, owing to the synchronization step at transition $z_3$.
To address this issue and still maintain the presence of an interlocutor with the patient and the brief oral summary told to physician, emergency doctors from the SAMU proposed to consider another model. One may create a new type of MRA, the {\em reservoir} assistant, who after a brief discussion with the MRA having answered the call, places the patient in a monitored reservoir. The answering MRA
is released to pick-up other inbound calls. When an emergency physician becomes available, the reservoir assistant passes on the short briefing to the doctor and transfers the patient.
While the queue of patients in the reservoir is non-empty, the reservoir assistant checks on the patients in the reservoir, and can call patients back in case they hung up.
This replaces the synchronizations between physicians and answering MRAs,
enabling the latter to pick-up new calls more quickly.
Another advantage of the reservoir mechanism is that
if a single reservoir assistant is sufficient to handle all the calls, this agent can have a consolidated vision of all the patients waiting for emergency physicians and revise in real time their priority level if more severe cases arrive, whereas the emergency physician may previously have had to ask each of the waiting MRAs.

The model~\eqref{eq:SAMU2}, whose dynamics shall be introduced and studied in~\Cref{sec:casestudyprio}, implements these modifications; see~\Cref{fig:SAMU2}. The reservoir assistant pool is a new place with initial marking $N_R$ (not necessarily equal to $1$). Reservoir assistants receive patients from the answering MRAs at transition $z_3$ and pass them to physicians at transitions $z_5$ and $z_5'$, depending on the severity of the case. We denote by $\alpha$ the proportion of very urgent calls among patients who need to talk to an emergency physician. In case of conflict, reservoir assistants must first pass the calls already in the reservoir before placing other calls in, and should first handle very urgent calls. Release of the reservoir assistants happen at transitions $z_4$, $z_6$ and $z_6'$. Consultations with a physician take a time $\tau_3$ after which transitions $z_7$ and $z_7'$ can be fired.
The circuits involving the reservoir assistant are depicted with color orange. It can be verified that the places standing for the pool of reservoir assistants and physicians have compatible priority rules.

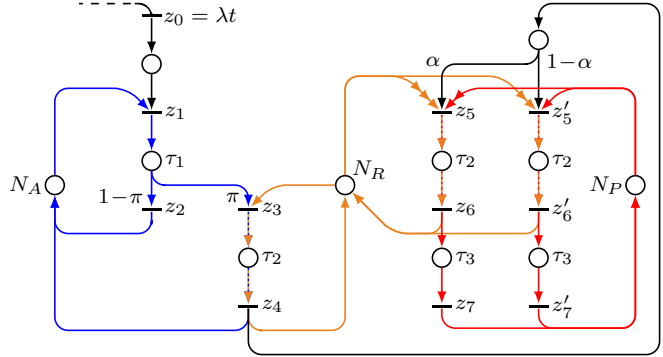
\begin{figure}[t]
\centering
\def\tkzscl{.29}
 % emergency_PN_usecase2bis.tex

\definecolor{colorARM}{rgb}{0,0,1}
\definecolor{colorARMres}{rgb}{.92,.5,.11}
\definecolor{colorAMU}{rgb}{1,0,0}
\definecolor{colorexit}{rgb}{0.4,0.4,0.4}

\tikzset{place/.style={draw,circle,inner sep=2.5pt,semithick}}
\tikzset{transition/.style={rectangle, thick,fill=black, minimum width=2mm,inner ysep=0.5pt, minimum width=2mm}}
\tikzset{jeton/.style={draw,circle,fill=black!80,inner sep=.35pt}}
\tikzset{pre/.style={=stealth'}}
\tikzset{post/.style={->,shorten >=1pt,>=stealth'}}
\tikzset{-|/.style={to path={-| (\tikztotarget)}}, |-/.style={to path={|- (\tikztotarget)}}}
\tikzset{bicolor/.style 2 args={dashed,dash pattern=on 1pt off 1pt,#1, postaction={draw,dashed,dash pattern=on 1pt off 1pt,#2,dash phase=1pt}}}
\tikzset{arrowPetri/.style={>=latex,rounded corners=5pt,semithick}}

\begin{tikzpicture}[scale=\tkzscl,font=\scriptsize]

\def\p{2.2}

\begin{scope}[shift={(0,0)}]
\node[place] (pool_arm) at ($(-2*\p,-1.5*\p)$) {};
\node (txt_Na) at ($(pool_arm)+(-1.3,0)$) {$N_A$};

\node[transition]    (q_arrivals)                          at    (0, 0) {};
\node[place]         (p_inc_calls)                         at    ($(q_arrivals) + (0, \p)$)  {};
\node[transition]    (q_inc_calls)                         at    ($(p_inc_calls) + (0, \p)$)  {};

\node[place]         (p_arrivals)        at    ($(q_arrivals) + (0,-\p)$) {};
\node[transition]    (q_debut_NFU)       at    ($(p_arrivals) + (0,-\p)$) {};
\node[transition]    (q_synchro)         at    ($(p_arrivals) + (2*\p,-\p)$) {};

\node[place]         (pool_res)          at    ($(q_synchro) + (2*\p,.5*\p)$) {};
\node (txt_Nr) at ($(pool_res)+(1.1,.7)$) {$N_R$};

\draw[->,arrowPetri,colorARM] (p_arrivals)  |- ($(p_arrivals)+(.5*\p,-.5*\p)$) -| (q_synchro);
\draw[->,arrowPetri,colorARM] (q_debut_NFU) |- ($(q_debut_NFU)+(0,-.5*\p)$) -| (pool_arm);
\draw[->,arrowPetri,colorARM] (p_arrivals)  -- (q_debut_NFU);
\draw[->,arrowPetri,colorARM] (q_arrivals) -- (p_arrivals);

\node[place]         (p_synchro)           at    ($(q_synchro)  + (0,-\p)$) {};
\node[transition]    (q_unsynchro)         at    ($(p_synchro)  + (0,-\p)$) {};

\node[place]      (pool_amu)    at    ($(pool_res)+(6*\p,0)$) {};
\node (txt_Nm) at ($(pool_amu)+(-1.3,0)$) {$N_P$};

\node[place]         (p_synchro2)          at    ($(pool_amu)  + (-4*\p,.5*\p)$) {};
\node[transition]    (q_unsynchro2)        at    ($(p_synchro2)  + (0,-\p)$) {};
\node[transition]    (q_debut_AMU)         at    ($(p_synchro2) + (0,\p)$) {};
\node[place]         (p_waiting)           at    ($(q_debut_AMU)  + (2*\p,1.5*\p)$) {};
\node[place]         (p_consult_AMU)       at    ($(q_unsynchro2)+ (0,-1*\p)$) {};
\node[transition]    (q_end_consult_AMU)   at    ($(p_consult_AMU)+(0,-\p)$) {};

\node[place]         (p_synchro3)          at    ($(p_synchro2)       + (2*\p,0)$) {};
\node[transition]    (q_unsynchro3)        at    ($(q_unsynchro2)     + (2*\p,0)$) {};
\node[transition]    (q_debut_AMU_2)       at    ($(q_debut_AMU)      + (2*\p,0)$) {};
\node[transition]    (q_end_consult_AMU_2) at    ($(q_end_consult_AMU)+ (2*\p,0)$) {};
\node[place]         (p_consult_AMU_2)     at    ($(p_consult_AMU)    + (2*\p,0)$) {};

\draw[arrowPetri]    (q_inc_calls) |- ($(q_inc_calls)+(-.35*\p,.25*\p)$);
\draw[dashed, arrowPetri]    (q_inc_calls) |- ($(q_inc_calls)+(-1.5*\p,.25*\p)$);
\draw[->,arrowPetri] (q_inc_calls) -- (p_inc_calls);
\draw[->,arrowPetri]                                    (p_inc_calls) -- (q_arrivals);
\draw[->,arrowPetri,colorARM]                           (q_unsynchro) -- ($(q_unsynchro)+(0,-.5*\p)$)  -|  (pool_arm);

\draw[->,arrowPetri,bicolor={colorARM}{colorARMres}]    (q_synchro) -- (p_synchro);
\draw[->,arrowPetri,bicolor={colorARM}{colorARMres}]    (p_synchro) -- (q_unsynchro);

\draw[->,arrowPetri,colorARMres]    (pool_res) -- ($(q_synchro)+(.5*\p,.5*\p)$) -- (q_synchro);
\draw[->,arrowPetri,colorARMres]    (q_unsynchro) |- ($(q_unsynchro)+(.5*\p,-.5*\p)$)-| (pool_res);

\draw[->>>,arrowPetri,colorARMres]   (pool_res) |- ($(q_debut_AMU)+(-.75*\p,.75*\p)$) -- (q_debut_AMU);
\draw[->,arrowPetri,colorARMres]    (q_unsynchro2) |- ($(q_unsynchro2)+(-\p,-.5*\p)$) -- (pool_res);

\draw[->>,arrowPetri,colorARMres]  (pool_res) |- ($(q_debut_AMU_2)+(-.75*\p,.75*\p)$) -- (q_debut_AMU_2);
\draw[->,arrowPetri,colorARMres]    (q_unsynchro3) |- ($(q_unsynchro2)+(-\p,-.5*\p)$) -- (pool_res);

\draw[->,arrowPetri,bicolor={colorAMU}{colorARMres}]    (q_debut_AMU) -- (p_synchro2);
\draw[->,arrowPetri,bicolor={colorAMU}{colorARMres}]    (p_synchro2) -- (q_unsynchro2);
\draw[->,arrowPetri,bicolor={colorAMU}{colorARMres}]    (q_debut_AMU_2) -- (p_synchro3);
\draw[->,arrowPetri,bicolor={colorAMU}{colorARMres}]    (p_synchro3) -- (q_unsynchro3);

\draw[->,arrowPetri]    (p_waiting) |- ($(p_waiting)+(-.5*\p,-.5*\p)$) -| (q_debut_AMU);
\draw[->,arrowPetri]    (p_waiting) -- (q_debut_AMU_2);
\draw[->,arrowPetri]    (q_unsynchro) |- ($(q_unsynchro)+(1.5*\p,-1*\p)$) -| ($(pool_amu)+(.5*\p,0)$) |- ($(p_waiting)+(0,.75*\p)$) -- (p_waiting);

\draw[->,arrowPetri,colorARM]                           (pool_arm)    |- ($(q_arrivals)+(-.5*\p,.5*\p)$) -- (q_arrivals);

\node (txt_taus) at ($(p_synchro)+(.5*\p,0)$) {$\tau_2$};
\node (txt_taus) at ($(p_synchro2)+(.5*\p,0)$) {$\tau_2$};
\node (txt_taus) at ($(p_synchro3)+(.5*\p,0)$) {$\tau_2$};
\node (txt_tau1) at ($(p_arrivals.center)+(.5*\p,0)$) {${\tau}_1$};
\node (txt_tau3) at ($(p_consult_AMU.center)+(.5*\p,0)$) {${\tau}_3$};
\node (txt_tau3) at ($(p_consult_AMU_2.center)+(.5*\p,0)$) {${\tau}_3$};
\node (txt_pi)  at ($(p_arrivals.center)+(-1.4,-1.65)$) {$1\!-\!\pi$};
\node (txt_pi3) at ($(p_arrivals.center)+(3.7, -1.65)$){$\pi$};
\node (txt_alpha3) at ($(p_waiting.center)+(-4.8,-1)$){$\alpha$};
\node (txt_alpha)  at ($(p_waiting.center) +(1.4,-1)$) {$1\!-\!\alpha$};

\draw[->,arrowPetri,colorAMU]    (q_unsynchro2)      -- (p_consult_AMU);
\draw[->,arrowPetri,colorAMU]    (p_consult_AMU)     -- (q_end_consult_AMU);
\draw[->,arrowPetri,colorAMU]    (q_end_consult_AMU) -- ($(q_end_consult_AMU)+(0,-1)$) -| (pool_amu);
\draw[->>,arrowPetri,colorAMU]   (pool_amu)      |- ($(q_debut_AMU)+(.5*\p,.5*\p)$)      -- (q_debut_AMU);

\draw[->,arrowPetri,colorAMU]    (q_unsynchro3)       -- (p_consult_AMU_2);
\draw[->,arrowPetri,colorAMU]    (p_consult_AMU_2)     -- (q_end_consult_AMU_2);
\draw[->,arrowPetri,colorAMU]    (q_end_consult_AMU_2) -- ($(q_end_consult_AMU_2)+(0,-1)$) -| (pool_amu);
\draw[->,arrowPetri,colorAMU]    (pool_amu)      |- ($(q_debut_AMU_2)+(.5*\p,.5*\p)$)      -- (q_debut_AMU_2);

\node (txt_z0) at ($(q_inc_calls)+(1*\p,0)$) {$z_0 = \lambda t$};
\node (txt_z1) at ($(q_arrivals)+(.5*\p,0)$) {$z_1$};
\node (txt_z2) at ($(q_debut_NFU)+(.5*\p,0)$) {$z_2$};
\node (txt_z3) at ($(q_synchro)+(.5*\p,0)$) {$z_3$};
\node (txt_z4) at ($(q_unsynchro)+(.5*\p,0)$) {$z_4$};
\node (txt_z5) at ($(q_debut_AMU)+(.5*\p,0)$) {$z_5$};
\node (txt_z6) at ($(q_unsynchro2)+(.5*\p,0)$) {$z_6$};
\node (txt_z7) at ($(q_end_consult_AMU)+(.5*\p,0)$) {$z_7$};
\node (txt_z5) at ($(q_debut_AMU_2)+(.5*\p,0)$) {$z_5'$};
\node (txt_z6) at ($(q_unsynchro3)+(.5*\p,0)$) {$z_6'$};
\node (txt_z7) at ($(q_end_consult_AMU_2)+(.5*\p,0)$) {$z_7'$};

\end{scope}

\end{tikzpicture}

\caption{Medical emergency call center with a monitored reservoir (EMS-B)}
\label{fig:SAMU2}
\vspace{-.55cm}
\end{figure}

\section{Basic definitions and tools for SMDPs}

Although Markov Decision Processes are classical in control theory and stochastic processes, the semi-Markov case is more delicate: we recall in this section several results concerning Markov chains and semi-Markov decision processes needed in~\Cref{sec:correspondence}. The reader already familiar with this framework may skip this part.

Recall that {\em Markov Decision Processes} (MDPs) form a class of one-player games, in which one evolves through \emph{states} by choosing \emph{actions}  at \emph{discrete time instants}, which determine some {\em costs}.
{\em Semi-Markov Decision Processes} (SMDPs, or Markov renewal programs) allow the time to take real values, while the state space remains discrete: between two successive moves, a \emph{holding time} attached to states and actions must elapse.
We refer for instance to~\cite{puterman2014markov,yushkevich1982semi} for in-depth background.

The finite set of states is denoted by $S$, and for all $i\in S$ the finite set of playable \emph{actions} from state $i$ is denoted by $A_i$.
 We denote $A \coloneqq \bigcup_{i\in S}A_i$.
 As a result of playing action $a$ from state $i$, the player incurs a cost $r^a_i$,
is held in the state $i$ for a non-negative time $t_i^a$, and finally
goes to state $j\in S$ with probability $P^a_{ij}$ (it is assumed that
$\sum_{j\in S}P^a_{ij}=1$ for all $i\in S$ and $a\in A_i$).
Moreover, future costs are multiplied by a discount factor $\gamma^{a}_i\geq 0$.
A common choice is $\gamma^{a}_i=e^{-\alpha t_i^{a}}$ (with $\alpha>0$) to reflect time preference, though we shall also allow to take $\gamma^{a}_i\geq 1$.

\subsection{Definition of the value function}
\label{sec:defSMDP}

A {\em history of length $n$} of the process
is a sequence $h_n = i_0,a_0,i_1,a_1,\dots i_n$, where for all $0\leq k\leq n$, $a_k\in A_{i_k}$. We denote by $H$ the set of all histories of finite lengths. A {\em strategy} $f$ is a map from $H$ to $\bigcup_{i\in S}\Delta(A_i)$ (where $\Delta(X)$ denotes the set of probability measures over the set $X$) such that $f(h_n)\in \Delta(A_{i_n})$. A strategy $f$ is called {\em Markovian} if $f(h_n)$ depends only on the current state $i_n$, {\em deterministic} if $f(h_n)$ is a Dirac measure on $A_{i_n}$, and {\em stationary} if it does not depend on the epoch $n$. A strategy $f$ and an initial state $i\in S$ define a probability measure $\mathbb{P}^f_i$ on $H$. If $h\in H$ is a history of length $n$ and $k\leq n$ ($k\in\mathbb{N}$), we denote by $\widehat{r_k}$ (resp.\ $\widehat{t_k}$ and $\widehat{\gamma_k}$) the random variable from $H$ to $\R$ such that $\widehat{r_k}(h)=r^{a_k}_{i_k}$ (resp.\ $\widehat{t_k}(h)=t^{a_k}_{i_k}$ and $\widehat{\gamma_k}(h)=\gamma^{a_k}_{i_k}$).

A {\em policy} $\sigma$ is a map from $S$ to $A$ such that $\sigma(i)\in A_i$ for every state $i\in S$ (some authors refer to this object as a {\em decision rule}). A deterministic Markovian strategy can be identified to a sequence of policies, and to a single policy if it is also stationary.
If $\sigma$ is a policy, $P^{\sigma}$ denotes the $|S|\times|S|$ matrix with entries $(P^{\sigma(i)}_{ij})_{i,j\in S}$, while $r^{\sigma}$ (resp.\ $t^{\sigma}$ and $\gamma^{\sigma}$) is the vector with entries $(r^{\sigma(i)}_i)_{i\in S}$ (resp.\ $(t^{\sigma(i)}_i)_{i\in S}$ and $(\gamma^{\sigma(i)}_i)_{i\in S}$).

\medskip
The {\em value function} $v:S\times \R \to \R$ of the game in finite horizon is defined as follows, so that for $i\in S$ and $t\geq 0$, $v(i,t)$ denotes the minimum (over all strategies) expected cost incurred by the player up to time $t$ by starting in state $i$ (by convention, $v(\cdot,t)=0$ for $t<0$):
\begin{equation}
   v(i,t) \coloneqq \inf_{f} \;\mathbb{E}^{f}_{i}\sum_{k=0}^{\widehat{N_t}} \bigg(\prod_{\ell = 0}^{k-1}\widehat{\gamma_{\ell}}\bigg)\widehat{r_k}\label{e-defvalue}
\end{equation}
where $\mathbb{E}^{f}_{i}$ denotes the expectation operator relatively to $\mathbb{P}^f_i$, and $\widehat{N_t}$ is the random variable from $H$ to $\N$ such that $\widehat{N_t}(h)=\sup\;\{n\in\mathbb{N}\;|\; \sum_{k=0}^n \widehat{t_k}(h) \leq t\}$.

\medskip
Allowing moves with zero duration generally makes the
expectation in~\eqref{e-defvalue} ill-defined. Hence,
a restriction is in order. We associate with the SMDP a
directed graph keeping track of these moves: this graph
has node set $S$, with an arc from $i$ to $j$
whenever there is an action $a\in A_i$ such that
$P^a_{ij}>0$ and $t_i^a=0$.
We shall say that the SMDP is
{\em non-Zeno} if this graph is acyclic.
Then, the random variable $\widehat{N}_t$
in~\eqref{e-defvalue} is bounded by the ratio $|S|t/t^*$
where $t^*=\min\{t_i^a\mid i\in S, a\in A_i, t_i^a>0\}$,
which entails that the expectation in~\eqref{e-defvalue}
is well defined for any choice of discount factors $\gamma_i^a$.

The following theorem expresses that the value function follows a Bellman-type optimality equation, see for instance~\cite[\S2, p.~800]{yushkevich1982semi} where the undiscounted case is addressed:
\begin{theorem}
The value function satisfies the following dynamic programming equation :
\begin{equation}
    v(i,t) = \inf_{a\in A_i} \bigg\{ r^{a}_i + \gamma^{a}_i\sum_{j\in S} P^{a}_{ij}\,v(j, t-t^{a}_i) \bigg\} \, .
\label{eq:SMDP_finite_horizon}
\end{equation}
\end{theorem}

In the case where the game is played over an infinite horizon and the discount factors are strictly less than $1$ (yielding a short-sighted cost criterion), the limit $\lim_{t\to\infty}v(i,t)$ exists and occurs to be the natural criterion to minimize. However, in an undiscounted framework ($\gamma_i^a\equiv 1$) where the previous limit does not exist, the player will rather seek to minimize its \emph{average cost}, \ie~its ultimate mean loss per unit of time.

\medskip
The average cost criterion $g^*$ is then defined as:
\begin{equation}
g^*(i) \coloneqq \inf_{f} g^f(i) \qquad \text{where} \qquad g^f(i)\coloneqq \liminf_{t\to\infty} \;\frac{1}{t}\;\mathbb{E}^{f}_{i}\sum_{k=0}^{\widehat{N_t}} \widehat{r_k}
   \label{eq:avgcostdef}
\end{equation}

Note that other types of average costs criteria that differ from $g^*$ in the semi-Markov case may also be defined, see for instance~\cite{jianyong2004average}.

\subsection{Subchains analysis and results on the ergodic problem}
\label{ap:markovchains}

The average cost also satisfies recursive optimality equations that we recall in this subsection, and in addition is closely related to the subchain structure of the process. To that purpose, we recall the following definitions and results.

\begin{definition}[Spectral projector]
  Let $\sigma$ be a deterministic policy and $P^{\sigma}$ its associated probability matrix. The {\em spectral projector} $P^{\sigma,\star}$ is defined by:
  \[ P^{\sigma,\star}\coloneqq \lim_{n\to\infty}\frac{1}{n+1}\sum_{j=0}^n \,(P^{\sigma})^j\]
  \end{definition}
  This Ces{\`a}ro limit does exist because $1$ is the dominant eigenvalue of $P^{\sigma}$ and it is semisimple.
  Recall that a {\em class} is a strongly connected component in the digraph of $P^{\sigma}$, and that this class is {\em final} if its elements do not have access to any element of another class. Denoting by $F_1,\dots,F_{m(\sigma)}$ the final classes, whose collection is denoted $\Fcal(\sigma)$, and by $Q$ the transient states under the policy $\sigma$, the state space admits a partition $S=F_1\cup F_2\cup\dots\cup F_{m(\sigma)}\cup Q$, and
  up to relabeling the states, we can write
  \begin{equation*}
      P^{\sigma} = \begin{pmatrix}
      \;Q_0\; & \;Q_1\;   & \;Q_2\;    &  \cdots & \;Q_m\; \\[-.5ex]
        0 & P_1   &  0     &  \cdots & 0 \\[-.5ex]
        0 &   0   & P_2    &  \cdots & 0 \\[-1.7ex]
    \vdots & \vdots& \vdots &  \ddots & 0  \\[-.5ex]
    0     & 0     & 0      & \cdots  & P_m
  \end{pmatrix},
  \;\text{and therefore}\quad\!
  P^{\sigma,\star} = \begin{pmatrix}
      \;\; 0\;\; & \;Q_1^{\star}\;   & \;Q_2^{\star}\;    &  \cdots & \;Q_m^{\star}\; \\[-.5ex]
      \;\; 0\;\; & P_1^{\star}   &  0     &  \cdots & 0 \\[-.5ex]
      \;\; 0\;\; &   0   & P_2^{\star}    &  \cdots & 0 \\[-1.7ex]
      \;\; \vdots\;\; & \vdots& \vdots &  \ddots & 0  \\[-.5ex]
  \;\; 0\;\;     & 0     & 0      & \cdots  & P_m^{\star}
  \end{pmatrix},  \end{equation*}
where $(P_k)_{1\leq k\leq m(\sigma)}$, dependent on $\sigma$, is the collection of stochastic matrices obtained by considering irreducible subchains of $P^{\sigma}$.
  We will consider {\em invariant measures} of $P^\sigma$. The latter
  are represented by nonnegative (row) vectors $\mu$, of sum one, such that $\mu P^\sigma =\mu$. The set of invariant measures is a convex polytope
  whose extreme points are the invariant measures supported by final classes~\cite[Chap.~8, Th.~3.23]{bermanandplemmons}. We denote by $\mu^{\sigma}_F$ the unique invariant
  measure supported by class $F\in\Fcal(\sigma)$.
  Observe that for all $k\in\{1,\dots,m(\sigma)\}$, $P_k^{\star}$ has identical rows,
  that coincide with the restriction of $\mu^{\sigma}_{F_k}$ to
  $F_k$.
  The entries of $\mu^{\sigma}_F$ represent the long-run fraction of time passed in the different states, assuming the initial state is in $F$.
  It is also known that $\phi_{F,i}^{\sigma} \coloneqq \sum_{j\in F} P^{\sigma,\star}_{ij}$ gives the probability that the Markov chain obtained by applying policy $\sigma$ starting from state $i$ ultimately reaches class $F$.

  \begin{theorem}[Average cost optimality equations~\cite{SchweitzerFedergruenEquation}]
    \label{thm:ac}
  Suppose that for all deterministic policies $\sigma$ and for any final class $F$ in $\Fcal(\sigma)$, we have $\sum_{i\in F}t^{\sigma(i)}_i > 0$ (\ie~no subchain can be travelled through in zero time).
Then:

\begin{enumerate}[label=(\roman*)]
\item the minimal average cost $g^*$ is achieved by stationary policies and satisfies
\[
 g^* (i) = \min_{\sigma} \sum_{F\in\Fcal(\sigma)} \phi_{F,i}^{\sigma} \frac{\langle \mu^{\sigma}_F, r^{\sigma}\rangle}{\langle \mu^{\sigma}_F, t^{\sigma}\rangle}  \,,
\]
where the minimum is taken over all the policies,
\item the minimal average cost $g^*$ is the unique vector $g\in \R^{S}$ such that there exists a vector $h\in\R^{S}$ verifying
  for all $i\in S$:
\vspace{-.15cm}
\begin{align}
  \tag{OE1}
  g(i) &= \min_{a\in A_i}\Big\{ \sum_{j\in S}P^a_{ij}\, g(j)\Big\}
  \label{eq:OE1}
  \\
  \tag{OE2}
  h(i) &= \min_{a\in A_i^*}\Big\{ r^a_i - t^a_i g(i) + \sum_{j\in S} P^a_{ij}\, h(j) \Big\}
  \label{eq:OE2}
\end{align}
where $A_i^*$ is the subset of $A_i$ where the minimum is achieved in~\eqref{eq:OE1}.
\end{enumerate}
\end{theorem}

\begin{remark}
We may also write~\eqref{eq:OE1} and~\eqref{eq:OE2} in a more compact way using the lexicographic-order on real tuples (we say that $(g,h)\leq (g',h')$ iff $g\leq g'$ or $g=g'$ and $h\leq h'$).
Denoting by $\mathrm{diag}(v)$ the $n\times n$ diagonal matrix with coefficients $v_1,\dots, v_n$ for $v\in\R^n$, we have entrywise:
\end{remark}
\begin{equation}
  \tag{OE}
  \label{eq:OE}
  (g,h) = \min_{\sigma}\!{}^{\textsc{lex}}\big\{ \big(P^{\sigma}\, g,r^{\sigma} + P^{\sigma}h - \mathrm{diag}(t^{\sigma})P^{\sigma}g\big)\big\}.
\end{equation}

\begin{remark}
A common special case of the SMDP problem
assumes that $\gamma_i^a\leq 1$ for all $i\in S$ and $a\in A_i$.
Then, the non-Zeno assumption that we made is too strong:
for the value in~\eqref{e-defvalue}
to be well-defined, it suffices to assume that for all policies $\sigma$,
there is at least one state with positive holding time in every final class
of $\Fcal(\sigma)$ just like in Theorem~\ref{thm:ac} (see also~\cite{Denardo-Fox68}).
\end{remark}

\section{Tools from nonexpansive mappings}\label{sec:tools}

The dynamics of timed Petri nets {\em without priority
  rules} have good features, which are best understood
as special cases of abstract properties of nonexpansive or order-preserving
mappings, which we recall, or establish in this section.

\medskip
We say that a self-map $F$ of a normed space $(X,\|\cdot\|)$ is {\em nonexpansive} if
\[ \forall x,y\in X,\quad \|F(x)-F(y)\|\leq \|x-y\|\, . \]

If $F$ is a self-map of an ordered space $(X,\leq)$, we say that $F$ is
{\em order-preserving} if
\[ \forall x,y\in X,\quad x\leq y\implies F(x) \leq F(y) \,.\]

Suppose now that $K$ is a compact set and denote $X=\mathscr{C}(K)$ the set of continuous and real-valued functions on $K$.
Given a positive function $e\in \mathscr{C}(K)$, we denote by $\|\cdot\|_e$
the {\em weighted sup-norm} $\|x\|_e \coloneqq \max_{v\in K}|x(v)/e(v)|$.
We also define the {\em weighted Hilbert's seminorm} $\|x\|_{e,H}=\max_{v\in K}x(v)/e(v) - \min_{v\in K}x(v)/e(v)$.
We say that $F: \mathscr{C}(K) \to \mathscr{C}(K)$
is {\em additively homogeneous} with respect to the function $e$ if
\[ \forall x\in \mathscr{C}(K),\quad \forall\alpha\in\R,\quad F(x+\alpha e) = F(x) + \alpha e\,,\]
\ie~if it commutes with the addition of scalar multiples of $e$.
We equip $\mathscr{C}(K)$ with the standard partial order.
We shall consider in particular the situation in which $K$ is the finite set $\{1,\dots,n\}$ equipped with the discrete topology. Then, the elements of $\mathscr{C}(K)$ will be identified to vectors of $\R^n$.

The following observation, made by Crandall and Tartar,
will play a key role in what follows.
\begin{proposition}[Crandall and Tartar, \cite{crandall}]\label{prop:crandall}
  Suppose $F: \mathscr{C}(K) \to \mathscr{C}(K)$ is additively homogeneous with respect to a positive function $e$ of $\mathscr{C}(K)$.
Then, the following assertions are equivalent:
  \begin{enumerate}[label=(\roman*)]
  \itemsep=0.9pt
  \item $F$ is order-preserving;
    \item $F$ is nonexpansive in the weighted sup-norm $\|\cdot\|_e$.
  \end{enumerate}
\end{proposition}
It is also known that when these assertions hold, $F$ is nonexpansive in the weighted Hilbert's seminorm, see e.g.~\cite{agn12}.

\medskip
When $F: (X,\|\cdot\|)\to (X,\|\cdot\|)$ is nonexpansive, we define the {\em escape rate vector}
\[
\chi(F) \coloneqq  \lim_{k\to \infty} F^k(x)/k
\]
where $x\in X$ is chosen in an arbitrary way. Indeed, by nonexpansiveness, the existence and the value of $\chi(F)$ are independent on the choice of $x$.

%\medskip
The following theorem of Kohlberg identifies a situation in which the escape rate does exist:
\begin{theorem}[Kohlberg, {\cite{kohlberg80}}]\label{thm:kohlberg}
  Suppose $F: \R^n\to \R^n$ is piecewise affine and nonnexpansive (in any norm). Then, there exists vectors $u,\rho\in\R^n$ such that
  \[ F(u+s\rho)= u+(s+1) \rho \,\qquad \forall s\geq 0 \, .\]
\end{theorem}
When $F$ satisfies the assumptions of this theorem, it follows readily that
\[
\chi(F)= \lim_{k\to\infty} F^k(u)/k = \lim_{k \to\infty} (u+k\rho)/k = \rho \, .
\]
We shall refer to the map $s\mapsto u+s\rho$ (or to the pair $(u,\rho)$) as an {\em invariant half-line}.

\medskip
We shall see in Proposition~\ref{prop:unitary} that for a prototypical
class of timed Petri net without priorities, the counter
function $z(t)$ is governed by a dynamics of the form
$z(t)=F(z(t-1))$, where $F$ is order-preserving and non-expansive in a weighted sup-norm. Then,
the escape rate vector coincides with the {\em throughput vector}
$\lim_{t\to \infty}z(t)/t$, which yields the average
number of firings, per time unit, of the different transitions.
Hence, Theorem~\ref{thm:kohlberg} will entail
that the throughput vector does exist.

We are now interested in finer results, concerning the deviation
$F^k(x) -k\chi(F)$ as $k\to\infty$, and in particular, its convergence to a periodic regime.
First, the next theorem
addresses the case in which $\chi(F)=0$. We denote by $\mathrm{Sym}(n)$ the symmetric group (set of permutations) on $n$ letters.

\begin{theorem}[see~\cite{martus,LN12}, \cite{LS2} and~\cite{spectral}]\label{thm:period}
  Suppose that $F:\R^n\to \R^n$ is nonexpansive in a polyhedral norm
  and that the orbits $\{ F^k(x) \;,\; k\in\Z\}$ of $F$ are bounded for all $x\in\R^n$. Then, for every $x\in \R^n$,
  there exists an integer $c$ bounded only as a function of the norm
  such that $F^{kc}(x) $ converges as $k\to \infty$. Moreover,
  if $F$ is order-preserving and weighted-sup-norm nonexpansive, then
  $c\leq {n \choose{\lfloor n/2\rfloor }}$. If in addition $F$ is concave,
  then $c$ is the order of an element of $\mathrm{Sym}(n)$.
\end{theorem}

The first part of the theorem was proved in~\cite{martus}
and in several other works, see the discussion in~\cite{LN12}.
The bound of $c$ in the order-preserving and sup-norm nonexpansive case
is established in~\cite{LS2}. The bound
in the concave case is established in~\cite{spectral}.

The following result deals with a special case of dynamics
with orbits that are bounded in Hilbert's seminorm.
This entails in particular that $\chi(F)$ is a scalar multiple of $e$.
This typically occurs in the theory
of \emph{unichain} Markov decision processes
(\ie~every policy admits a single recurrent class).

\begin{theorem}\label{th-general-periodic}
  Suppose that $F:\R^n\to \R^n$ is order-preserving and additively
  homogeneous. Suppose in addition that the sequence
  of Hilbert's seminorms $(\|F^k(x)\|_{e,H})_{k\in\N}$ is bounded. Then, for all
  $x\in\R^n$, there
  exists an integer $c$ such that
  for all $0\leq r\leq c-1$, $F^{kc+r}(x) - \chi(F) (kc+r)$ converges
  as $k\to\infty$. Moreover, $c$ can be bounded as in Theorem~\ref{thm:period}.
\end{theorem}
\begin{proof}
It follows from \cite{GauGu} that $F$ has an additive eigenvector, meaning
that there exists $u\in \R^n$ and $\lambda\in \R$ such that $F(u)=u+\lambda e$.
Then, the map $G \coloneqq F-\lambda e$ has a fixed point, and it is still
order-preserving and sup-norm nonexpansive. It follows
from Theorem~\ref{thm:period} that for every $x\in \R^n$,
  there exists an integer $c$
  such that $G^{kc}(x) $ converges as $k\to \infty$. Since
  $G$ is continuous, $G^{kc+r}(x)=F^{kc+r}(x)-(kc+r)\lambda e=
  F^{kc+r}(x)-(kc+r)\rho$ also converges as $k\to\infty$.
  The bounds on $c$ follow from the one of Theorem~\ref{thm:period}.
\end{proof}

The next theorem is
stated in~\cite{schweitzer79} for operators of multichain
Markov decision processes. We provide a more abstract (equivalent)
statement.
\begin{theorem}[Schweitzer and Federgruen, \cite{schweitzer79}]\label{thm:schweitzer}
  Suppose that $F$ is concave, order-preserving, additively homogeneous,
  and piecewise linear.
  Then, for all $x\in \R^n$, there exists an integer $c$ that is the order of
  an element of $\mathrm{Sym}(n)$, such that for
  all $0\leq r\leq c-1$, the sequence $F^{kc+r}(x) -(kc+r)\chi(F)$
  converges as $k\to\infty$.
\end{theorem}

By comparison with Theorem~\ref{th-general-periodic},
the map $F$ is required in addition to be {\em concave}
and {\em polyhedral}. The concavity assumption
leads to a refined explicit formula on the period $c$ (given by a combinatorial invariant
of a certain critical graph depending only on optimal stationary
randomized policies, see~\cite{spectral}). The polyhedrality assumption allows one to avoid the restriction to maps whose orbits are bounded in Hilbert's seminorm.

\section{Correspondence between fluid Petri nets and Semi-Markov Decision Processes}

In this section, we establish formal correspondences between Petri net dynamics and Bellman equations' of SMDPs, and derive or refine several results on the asymptotic throughputs of Petri nets transitions.
\label{sec:correspondence}

\subsection{The correspondence theorems}

\begin{correspondence}
  \label{thm:correspondence1}
  Consider a timed Petri net with no priority rules.
  Then, its dynamics is equivalent
  to the dynamic programming equation of a
  semi-Markov decision process with controlled discount factors.
\end{correspondence}
\begin{proof}
We extend the definition of proportions $\pi_{qp}$ by letting $\pi_{qp}=1$ if $q\in \Qsync$ and $p\in q\inc$. Similarly, we set the weights $\alpha_{qp}$ (resp.\ $\alpha_{pq}$) to $0$ if $p \notin q\inc$ (resp. $q \notin p\inc$).
For all $q,q'\in \Qcal$ and $p\in \Pcal$, we set
\[
c^p_q \coloneqq
\begin{cases}
\pi_{qp}\alpha_{qp}^{-1} m_p & \text{if} \; p\in q\inc  \\
0 & \text{otherwise,}
\end{cases}
 \qquad\text{and}\qquad
\tilde{\beta}^p_{qq'} \coloneqq
\begin{cases}
\pi_{qp}\alpha_{qp}^{-1} \alpha_{pq'} & \text{if} \; p\in q\inc \\
0 & \text{otherwise.}
\end{cases}
\]
By definition, for $q\in\Qcal$ and $p\in q\inc$, the nonnegative numbers $(\tilde{\beta}^p_{qq'})_{q'\in\Qcal}$ are not all zero. We let $\kappa^p_q \coloneqq \sum_{q'\in\Qcal}\tilde{\beta}^p_{qq'}$ and $\beta^p_{qq'} = \tilde{\beta}^p_{qq'}/\kappa^p_q$ so that $(\beta^p_{qq'})_{q'\in\Qcal}$ is a probability vector. The dynamics summarized in Table 2 can then be written as
\begin{equation}
  \forall q\in\Qcal \quad z_q(t) = \min_{p\in q\inc} \bigg\{ c^p_q + \kappa^p_q \sum_{q'\in \Qcal}\beta^p_{qq'} z_{q'}(t-\tau_p) \bigg\}
\label{eq:petri_to_smdp}
\end{equation}
where we recognize the finite-horizon Bellman's equation of a discounted semi-Markovian decision process expressed in equation~(2).
\end{proof}
As we announced in~\Cref{table:correspondence}, the states of the SMDP built in the proof correspond to the transitions of the Petri net, and in each state $q\in \Qcal$ of the SMDP, the admissible actions are the upstream places $p\in q\inc$. After playing action $p$ from state $q$, the player incurs a cost $c^p_q = \pi_{qp}\alpha_{qp}^{-1} m_p$ and a discount factor $\kappa^p_q = \sum_{q' \in p\inc} \pi_{qp}\alpha_{qp}^{-1} \alpha_{pq'}$. Then, the player is held for time $\tau_p$, before moving to one of the states $q'\in p\inc$ with probability $\beta^p_{qq'} = \pi_{qp}\alpha_{qp}^{-1} \alpha_{pq'} / \kappa^p_q$. In other words, the physical time of timed Petri nets is the backward time (time remaining to live) in semi-Markov decision processes.
The other correspondences between these two families of model stated in~Table 1 shall be interpreted after our second Correspondence Theorem~6.3.

\medskip
Our analysis of the long-run behavior of the transitions of Petri nets
relies on the existence
of a stoichiometric invariant:
\begin{definition}
  \label{def:stoichio}
  We say
  a vector $(e_q)_{q\in\Qcal}$ is a
  {\em stoichiometric invariant} of the Petri net whose dynamics is given by~\eqref{eq:petri_to_smdp}
  if
  \begin{align} \forall q\in\Qcal, \quad \forall p\in q\inc, \qquad e_q = \kappa_q^p \sum_{q'\in\Qcal}\beta^p_{qq'}\,e_{q'} \;\;\bigg(\!\!=\pi_{qp}\alpha_{qp}^{-1}\sum_{q'\in p\inc}\alpha_{pq'}e_{q'}\bigg)\,.
    \label{e-def-stoichio}
    \end{align}
\end{definition}

Notice that this definition refines the notion of \emph{T-invariants}: recall that a T-invariant is a vector $y\in\R^{\Qcal}$ such that for all $p\in\Pcal$ we have $\sum_{q\in p\out}\alpha_{qp}y_q = \sum_{q'\in p\inc}\alpha_{pq'}y_{q'}$. Stoichiometric invariants are T-invariants and the two definitions coincide if for all $p\in q\inc$, $p\out = \{q\}$ ($q$ is the only downstream transition of its upstream places), however if $q$ is placed downstream to a place $p$ ruled by preselection routing, stoichiometric invariants express that not only the fluid flows of this place are balanced but also that the desired quantities of fluid (a proportion $\pi_{qp}$ of the total) goes through $q$.
As an illustration, it can be checked that $(1,1,1-\pi,\pi,\pi,\pi)$ is a stoichiometric invariant of our model~\eqref{SAMU1}, with transitions indices in $\{0,\dots,5\}$.

\begin{correspondence}
  \label{thm:correspondence2}
  Suppose there are no priority rules and that the Petri net admits a positive stoichiometric invariant $e$. Then, the dynamics of the timed Petri net is equivalent
  to the dynamic programming equation of an undiscounted
  semi-Markov decision process.
  \end{correspondence}
  \begin{proof}
    It suffices to observe that the transformed counters $\tilde{z}_q = z_q/e_q$
    follow an equation of type~\eqref{eq:SMDP_finite_horizon},
    with $P_{qq'}^p:=e_{q}^{-1}\kappa_q^p\beta_{qq'}^p e_{q'}$
    and $\gamma_q^p = 1$, thanks to~\eqref{e-def-stoichio}.\end{proof}

  \begin{figure}[h]
    \begin{center}
      \vspace{-.76cm}
      \begin{tabular}{ccc}\def\tkzscl{.27}
         % emergency_PN_usecase1_conflict_free.tex

\definecolor{colorARM}{rgb}{0,0,0}
\definecolor{colorARMres}{rgb}{0,0,0}
\definecolor{colorAMU}{rgb}{0,0,0}
\definecolor{colorexit}{rgb}{0,0,0}

\tikzset{place/.style={draw,circle,inner sep=2.5pt,semithick}}
\tikzset{transition/.style={rectangle, thick,fill=black, minimum width=2mm,inner ysep=0.5pt, minimum width=2mm}}
\tikzset{jeton/.style={draw,circle,fill=black!80,inner sep=.35pt}}
\tikzset{pre/.style={=stealth'}}
\tikzset{post/.style={->,shorten >=1pt,>=stealth'}}
\tikzset{-|/.style={to path={-| (\tikztotarget)}}, |-/.style={to path={|- (\tikztotarget)}}}
\tikzset{bicolor/.style 2 args={dashed,dash pattern=on 1pt off 1pt,#1, postaction={draw,dashed,dash pattern=on 1pt off 1pt,#2,dash phase=1pt}}}
\tikzset{arrowPetri/.style={>=latex,rounded corners=5pt}}

\begin{tikzpicture}[scale=\tkzscl,font=\tiny]

\def\p{2.2}

\begin{scope}[shift={(0,4)}]
\node[place] (pool_arm) at ($(-2*\p,-1.5*\p)$) {};
\node (txt_Na) at ($(pool_arm)+(-1.3,0)$) {$N_A$};

\node[transition]    (q_arrivals)                          at    (0, 0) {};
\node[place]         (p_inc_calls)                         at    ($(q_arrivals) + (0, \p)$)  {};
\node[transition]    (q_inc_calls)                         at    ($(p_inc_calls) + (0, \p)$)  {};
\node[place]         (p_arrivals)        at    ($(q_arrivals) + (0,-\p)$) {};
\node[place]         (p_arrivals2)       at    ($(q_arrivals) + (2*\p,-\p)$) {};
\node[transition]    (q_debut_NFU)       at    ($(p_arrivals) + (0,-\p)$) {};
\node[transition]    (q_debut_AMU)         at    ($(p_arrivals) + (2*\p,-\p)$) {};

\draw[->,arrowPetri,colorARM] (q_arrivals)  |- ($(q_arrivals)+(.5*\p,-.3*\p)$) -| (p_arrivals2);
\draw[->,arrowPetri,colorARM] (q_debut_NFU) |- ($(q_debut_NFU)+(0,-.5*\p)$) -| (pool_arm);
\draw[->,arrowPetri,colorARM] (p_arrivals)  -- (q_debut_NFU);
\draw[->,arrowPetri,colorARM] (q_arrivals) -- (p_arrivals);
\draw[->,arrowPetri,colorARM] (p_arrivals2) -- (q_debut_AMU);

\node[place]         (p_synchro)           at    ($(q_debut_AMU)  + (0,-\p)$) {};
\node[transition]    (q_unsynchro)         at    ($(p_synchro)  + (0,-\p)$) {};

\draw[arrowPetri]    (q_inc_calls) |- ($(q_inc_calls)+(-.35*\p,.25*\p)$);
\draw[dashed, arrowPetri]    (q_inc_calls) |- ($(q_inc_calls)+(-1.5*\p,.25*\p)$);
\draw[->,arrowPetri] (q_inc_calls) -- (p_inc_calls);
\draw[->,arrowPetri]                                    (p_inc_calls) -- (q_arrivals);
\draw[->,arrowPetri,colorARM]                           (q_unsynchro) -- ($(q_unsynchro)+(0,-.5*\p)$)  -|  (pool_arm);

\node[place]    (p_consult_AMU)    at    ($(q_unsynchro)+(0,-\p)$) {};

\draw[->,arrowPetri,bicolor={colorARM}{colorAMU}]    (q_debut_AMU) -- (p_synchro);
\draw[->,arrowPetri,bicolor={colorARM}{colorAMU}]    (p_synchro) -- (q_unsynchro);

\draw[->,arrowPetri,colorARM]                           (pool_arm)    |- ($(q_arrivals)+(-.5*\p,.5*\p)$) -- (q_arrivals);

\node (txt_taus) at ($(p_synchro)+(.5*\p,0)$) {$\tau_2$};
\node (txt_tau1) at ($(p_arrivals.center)+(.5*\p,0)$) {${\tau}_1$};
\node (txt_tau1) at ($(p_arrivals2.center)+(.5*\p,0)$) {${\tau}_1$};
\node (txt_tau3) at ($(p_consult_AMU.center)+(.5*\p,0)$) {${\tau}_3$};
\node (txt_pi)  at ($(p_arrivals.center)+(-1.4,1.65)$) {$1\!-\!\pi$};
\node (txt_pi3) at ($(p_arrivals.center)+(4.6, 1.65)$) {$\pi$};

\node[place]      (pool_amu_old)    at    ($(q_unsynchro)+(2*\p,.5*\p)$) {};
\node (txt_Nm) at ($(pool_amu_old)+(1.3,0)$) {$N_P$};

\node[transition] (q_end_consult_AMU)                  at    ($(p_consult_AMU)+(0,-\p)$) {};
\draw[->,arrowPetri,colorAMU]    (q_unsynchro)       -- (p_consult_AMU);
\draw[->,arrowPetri,colorAMU]    (p_consult_AMU)     -- (q_end_consult_AMU);
\draw[->,arrowPetri,colorAMU]    (q_end_consult_AMU) -- ($(q_end_consult_AMU)+(0,-1)$) -| (pool_amu_old);
\draw[->,arrowPetri,colorAMU]    (pool_amu_old)      |- ($(q_debut_AMU)+(.5*\p,.5*\p)$)      -- (q_debut_AMU);

\node (txt_z0) at ($(q_inc_calls)+(1*\p,0)$) {$z_0 = \lambda t$};
\node (txt_z1) at ($(q_arrivals)+(.5*\p,0)$) {$z_1$};
\node (txt_z2) at ($(q_debut_NFU)+(.5*\p,0)$) {$z_2$};
\node (txt_z3) at ($(q_debut_AMU)+(.5*\p,0)$) {$z_3$};
\node (txt_z4) at ($(q_unsynchro)+(.5*\p,0)$) {$z_4$};
\node (txt_z5) at ($(q_end_consult_AMU)+(.5*\p,0)$) {$z_5$};
\end{scope}
\end{tikzpicture}	

        & \hspace{.1cm} & \def\tkzscl{.55} % emergency_MDP.tex
\definecolor{orangeDIY}{rgb}{.92,.5,.11}
\begin{tikzpicture}[auto,node distance=8mm,>=latex,scale=\tkzscl]
  \tikzset{arrowPetri/.style={>=latex,rounded corners=5pt}}

    \tikzstyle{round}=[thick,draw=black,circle,inner sep=2pt]

    \def\p{2.2}

    \node[round] (s0) at (0,0)              {\scriptsize$0$};
    \node[round] (s1) at ($(s0)+(0,-\p)$) {\scriptsize$1$};
    \node[round] (s2) at ($(s1)+(0,-\p)$)   {\scriptsize$2$};
    \node[round] (s3) at ($(s1)+(\p,-\p)$)  {\scriptsize$3$};
    \node[round] (s4) at ($(s3)+(0,-\p)$)   {\scriptsize$4$};
    \node[round] (s5) at ($(s4)+(0,-\p)$)   {\scriptsize$5$};

   \coordinate (a0)  at ($(s0)+(0,.5*\p)$) {};
   \coordinate (a11) at ($(s1)+(0,.5*\p)$) {};
   \coordinate (a12) at ($(s1)+(-.5*\p,0)$) {};
   \coordinate (a2)  at ($(s2)+(0,.5*\p)$) {};
   \coordinate (a31) at ($(s3)+(0,.5*\p)$) {};
   \coordinate (a32) at ($(s3)+(.5*\p,0)$) {};
   \coordinate (a4)  at ($(s4)+(0,.5*\p)$) {};
   \coordinate (a5)  at ($(s5)+(0,.5*\p)$) {};

    \node (lbl_a0)  at ($(a0) +(.75,0)$) {\tiny$(\lambda,1)$};
    \node (lbl_a11) at ($(a11)+(.75,0)$) {\tiny$(0,0)$};
    \node (lbl_a12) at ($(a12)+(-.25,.4)$) {\tiny$(N_A,0)$};
    \node (lbl_a2)  at ($(a2) +(.85,0.2)$) {\tiny$(0,\tau_1)$};
    \node (lbl_a31) at ($(a31)+(.85,0.2)$) {\tiny$(0,\tau_1)$};
    \node (lbl_a32) at ($(a32)+(0.3,0.45)$) {\tiny$(N_P/\pi,0)$};
    \node (lbl_a4)  at ($(a4) +(-.9,0)$) {\tiny$(0,\tau_2)$};
    \node (lbl_a5)  at ($(a5) +(-.9,0)$) {\tiny$(0,\tau_3)$};
    \node (lbl_pi1) at ($(s2) +(-1.1,-1*\p+.3)$) {\tiny$\pi$};
    \node (lbl_pi1) at ($(s2) +(-1,.3)$) {\tiny$1\!-\!\pi$};

     \draw[-{Latex[scale=1]}, rounded corners=5pt] (a0) |- ($(a0)+(-.5*\p-.5,.2)$) |- (s0);
     \draw[-{Latex[scale=1]}, rounded corners=5pt] (a11) -- (s0);
     \draw[-{Latex[scale=1]}, rounded corners=5pt] (a12) -| ($(a12)+(-.5,-1)$) |- (s2);
     \draw[-{Latex[scale=1]}, rounded corners=5pt] (a12) -| ($(a12)+(-.5,-1)$) |- (s4);
     \draw[-{Latex[scale=1]}, rounded corners=5pt] (a31) |- (s1);
     \draw[-{Latex[scale=1]}, rounded corners=5pt] (a2) -- (s1);
     \draw[-{Latex[scale=1]}, rounded corners=5pt] (a4) -- (s3);
     \draw[-{Latex[scale=1]}, rounded corners=5pt] (a5) -- (s4);
     \draw[-{Latex[scale=1]}, rounded corners=5pt] (a32) -| ($(s5)+(.66*\p+.4,0)$) -- (s5);

   \draw[-{Square[scale=1.5, open, fill=white]}, semithick] (s0) -- (a0);
   \draw[-{Square[scale=1.5, open, fill=white]}, semithick] (s1) -- (a11);
   \draw[-{Square[scale=1.5, open, fill=white]}, semithick] (s1) -- (a12);
   \draw[-{Square[scale=1.5, open, fill=white]}, semithick] (s2) -- (a2);
   \draw[-{Square[scale=1.5, open, fill=white]}, semithick] (s3) -- (a31);
   \draw[-{Square[scale=1.5, open, fill=white]}, semithick] (s3) -- (a32);
   \draw[-{Square[scale=1.5, open, fill=white]}, semithick] (s4) -- (a4);
   \draw[-{Square[scale=1.5, open, fill=white]}, semithick] (s5) -- (a5);
\end{tikzpicture}
 \end{tabular}
    \vspace{-.3cm}
    \end{center}
    \caption{The (conflict-free) Petri net~\eqref{SAMU1} (left) and the corresponding undiscounted SMDP (right).}
    \label{fig:SMDP}   \vspace{-1mm}
    \end{figure}
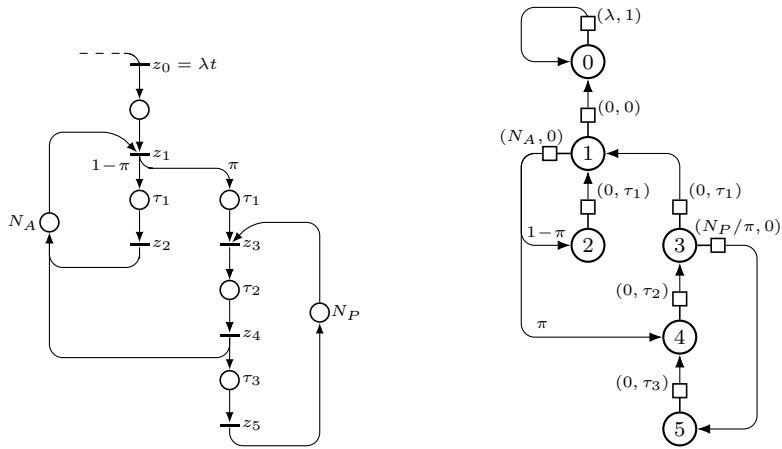

  We illustrate on~\Cref{fig:SMDP} the construction of the undiscounted SMDP corresponding to the Petri net~\eqref{SAMU1}. Actions (depicted by squares) are labeled by pairs consisting of the associated cost and holding time, and probabilities are given along the arcs from actions to states (when non equal to $1$). As we discussed after the proof of Correspondence Theorem~\ref{thm:correspondence1}, places and transitions of the Petri net are respectively mapped to the actions and states of the SMDP. The orientation of the arcs are therefore flipped (observe indeed how the expression $p\in q\inc$ corresponds to $a\in A_i$), and the time goes backward. For every state $q \in \Qcal$, the holding time of the action $p \in q\inc$ is $\tau_p$. Moreover, since all the arc weights in~\eqref{SAMU1} are $0$ or $1$, the cost $c_p^q$ reduces to $m_p / e_q$, which corresponds to a renormalization of the initial marking of the place $p$ by the stoichiometric coefficient of the transition~$q$.

  Notice that the obtained SMDP also corresponds to the \emph{conflict-free} version of the initial Petri net, which means that places ruled by preselection $p\in \Ppsel$ have been duplicated into $|p\out|$ new places $(p_q)_{q\in p\out}$ having only one downstream transition with weight $\alpha_{qp_q}=\alpha_{qp}$, and the probabilities $(\pi_{qp})_{q\in p\out}$ are shifted upwards, \ie~the previous upstream arcs of $p$ with weights $(\alpha_{pq})_{q\in p\inc}$ are replaced by $|p\out|$ more arcs with weights $(\alpha_{pq}\pi_{q'p})_{q\in p\inc, q'\in p\out}$. This transformation has been used by Gaujal and Giua in their work~\cite{gaujal2004optimal}, where they point out that it does not alter the stationary behavior since it leaves the dynamics equations unchanged.

\subsection{The evolution semigroup of the time-delay system}
\label{sec:semigroup}

In order to prevent an infinite number of firings from occurring in a finite amount of time, we shall work with Petri nets whose underlying directed graph does not contain any circuit in which places have zero holding times. Such Petri nets are said to be {\em non-Zeno}.

We first show that when the Petri net is priority-free and non-Zeno,
the counter variables are determined uniquely
by the dynamics of~\Cref{table:equation_counters}, given an initial condition.
\begin{lemma}\label{lemma:dyn-wellposed}
 Suppose that a Petri net is priority-free and non-Zeno, and let $T$ denote the maximum of the holding times of its different places. Then,
  the transition counter function $z : [-T,\infty) \to \R^\Qcal$, which follows the dynamics of \Cref{table:equation_counters}, is uniquely determined by its restriction to the interval $[-T,0]$.
\end{lemma}

\begin{proof}
  The counter functions satisfy a system of equations which is of the general form
 \begin{align}
  z_q(t)  = F_q( (z_{q'}(t-s))_{(q',s)\in U_q}) \label{e-general}
  \end{align}
 where $U_q$ is a finite subset of $\Qcal\times [0,T]$. Moreover,
 the pairs $(q',s)\in U_q$ are such that there is directed path from $q'$ to $q$ (here of length 2, since $q'\in p\inc$
 for some $p\in q\inc$).

 We first show that we can reduce to a dynamics
 of the form~\eqref{e-general} in which all the delays $s$ arising
 in the right-hand side are positive,
 by considering the following substitution procedure.
 If one variable $z_{q'}(t-s)$ with $s=0$ arises
 at the right-hand side of~\eqref{e-general}, we may replace
 this ocurrence of $z_{q'}(t)$
 using the relation $z_{q'}(t) = F_{q'}((z_{q''}(t-s))_{(q'',s)\in U_{q'}}$.
 We arrive at another expression of $z_q(t)$,
 still of the form~\eqref{e-general} with a modified set $U_q$,
   where this time, for all $(q',s)\in U_q$, $q'$
   is connected to $q$ by a directed path of increased length and $s$ is the
   sum of the holding times of the places in this path.
   For $\Qcal$ is finite, only a finite sequence of such substitutions can be performed, otherwise we would have some $q\in \Qcal$ such that $z_q(t)$ is substituted twice, providing a circuit of the net with only places with zero holding time, contradicting the non-Zeno assumption.
Finally, defining $\tau^* \coloneqq \inf \{\tau_p \colon \tau_p > 0 \, , \; p \in \Pcal\}$, Lemma~\ref{lemma:dyn-wellposed} is proved for all $t \in {[-T, n\tau^*)}$ by  induction on $n \geq 0$.
\end{proof}

In the rest of the section, the following assumption is made:

\begin{assumption}The Petri net is non-Zeno, has no priority rules and admits a positive stoichiometric invariant $e$.
  \label{ass:A}
\end{assumption}

The following immediate proposition derived from Lemma~\ref{lemma:dyn-wellposed} shows the nature
of Petri net dynamics in a remarkable special case.
\begin{proposition}\label{prop:unitary}
  Suppose Assumption~\ref{ass:A} is satisfied and that the holding times are all equal to $1$. Then, the dynamics of the Petri net,
\eqref{eq:petri_to_smdp}, can be rewritten as
\[ z(t)=F(z(t-1))
\,,
\]
where $F:\R^{\Qcal}\to \R^{\Qcal}$ is monotone, concave, and piecewise affine.
Moreover, $F$ is additively homogeneous with respect to $e$.
\end{proposition}

Under the conditions of Proposition~\ref{prop:unitary}, it follows from Proposition~\ref{prop:crandall} that $F$ is nonexpansive
with respect to the weighted sup-norm $\|x\|_e:=\max_{q\in\Qcal}|x_q/e_q|$.

The next result deals with the extension to the case where holding times are integer.

\begin{corollary}\label{coro:reduc-to1}
  Suppose  Assumption~\ref{ass:A} is satisfied and that the holding times are integer.
  Let $T$ be the maximal holding time. Then, there exists
  a concave and order-preserving piecewise affine self-map $F$ of $\R^{\Qcal \times \{1,\dots,T\}}$
  , such that the vector
    $\tilde{z}(t)=(z(t),\dots,z(t-T+1))$
    satisfies  $\tilde{z}(t)=F(\tilde{z}(t-1))$. In addition, $F$ is additively homogeneous
  with respect to the vector $(e_q)_{q\in \Qcal, t\in \{1,\dots,T\}}$.
\end{corollary}
\begin{proof}
  We apply the substitution procedure already
  used in the proof of~Lemma~\ref{lemma:dyn-wellposed}.
  This procedure allows us to replace
  the dynamics~\eqref{eq:petri_to_smdp}
  by a dynamics of the same form
  in which only the entries of $z(t-1),\dots,z(t-T)$
  occur at the right-hand side. Moreover,
  the class of concave, order-preserving
  piecewise-affine maps is preserved under
  substitutions of this nature. The additive
  homogeneity property is immediate.
\end{proof}

In contrast, when the holding times take irrational values,
the Petri net equations~\eqref{eq:petri_to_smdp} yield
a time delay system with a state space of {\em infinite dimension}.
To extend the previous approach, we need to represent
the evolution of this time-delay system by a semi-group.
We denote by $\Zcal=\mathscr{C}([-T,0])$ the space of continuous functions over $[-T,0]$.
The next proposition, which follows from~Lemma~\ref{lemma:dyn-wellposed}, ensures that it is well-posed to represent the evolution of counter functions by a one-parameter semigroup
$(\semigroup_t)_{t\geq 0}$ acting on $\ZQcal$, i.e., by a family of
self-maps of $\ZQcal$ satisfying
$\semigroup_{t_1+t_2}=\semigroup_{t_1}\circ \semigroup_{t_2}$.

\begin{proposition}
Suppose that a Petri net is priority-free and non-Zeno. The family of operators $\semigroup_t$, acting on $\ZQcal$,
   which associate with the function $z^0: s\mapsto (z^0_{q}(s))_{q\in \Qcal}$
      defined for $s\in [-T,0]$,
      the function $s \mapsto (z_q(s+t))_{q\in\Qcal}$ where
      $z$ is the solution of the dynamics determined by the initial condition
      $z^0$, constitutes a one-parameter semigroup.
\end{proposition}

     When in addition the Petri net admits a stoichiometric invariant $e$, we equip the infinite dimensional space $\ZQcal$
   with the weighted sup-norm:
     \[
     \|\varphi\|_e=\max_{q\in \Qcal} \sup_{s\in [-T,0]}\left|\frac{\varphi_q(s)}{e_q}\right|
     \, .\]
     This is consistent with the definition of the weighted sup-norm
     introduced in~\Cref{sec:tools}, identifying $\ZQcal$ with
     $\mathscr{C}([-T,0]\times \Qcal)$, and denoting
    by the same symbol $e$ a vector in $\R^{|\Qcal|}$ and the function $(s,q)\mapsto e_q$ in $\mathscr{C}([-T,0]\times \Qcal)$.
    The next result shows that our evolution semigroup, which is time-invariant, satifies additional good properties introduced in~\Cref{sec:tools}.

    \begin{proposition}\label{prop:nexp}
      Under Assumption~\ref{ass:A}, for all $t\geq 0$,
      the operator $\semigroup_t:\ZQcal\to \ZQcal$
      is order-preserving, additively homogeneous
      with respect to the function $e$, and nonexpansive with respect
      to the weighted sup-norm $\|\cdot\|_e$ on $\ZQcal\simeq \mathscr{C}([-T,0]\times \Qcal)$.
    \end{proposition}
    \begin{proof}
      The order-preserving and additive homogeneity of $\semigroup_t$
      follow from the fact
      that a trajectory $z$ is uniquely determined by its values
      on $[-T,0]$ (see Lemma~\ref{lemma:dyn-wellposed}), and from the order-preserving and homogeneity properties
      of the equation defining the dynamics. The nonexpansive property
      follows from~Proposition~\ref{prop:crandall}.
      \end{proof}

    \subsection{Existence and universality of the throughput}
    \label{ssec:throughput}

We are interested in the long-run time behavior of Petri nets. For this purpose, we introduce a notion of affine stationary regime.
\begin{definition}
  We say that a trajectory $z$ (counter functions of the transitions) of the Petri net is an {\em affine stationary regime} if there exists two vectors
  $\rho\in(\Rplus)^{\Qcal}$ and $u\in\R^{\Qcal}$ such that for all $t\geq -T$, $z(t) = \rho t + u$.
\end{definition}

The next proposition shows that, up to a shift in time, affine stationary regimes are characterized by a lexicographic system.
\begin{proposition}  \label{prop-lex}
  Suppose the Petri net has no priority rule. Given $\rho\in(\Rplus)^{\Qcal}$ and $u\in\R^{\Qcal}$,
  there exists a nonnegative number $t_0$ such that $z(t) \coloneqq \rho(t+t_0)+u$ is a stationary regime if and only if
\begin{align}
  \tag{L1}
  \rho_q &= \min_{p\in q\inc }\Big\{ \kappa_q^p\sum_{q'\in \Qcal} \beta_{qq'}^p\, \rho_{q'}\Big\}
  \label{eq:L1}
  \\
  \tag{L2}
  u_q &= \min_{p\in q\inc_{*}}\Big\{ c^p_q - \rho_q\tau_p + \kappa_q^p\sum_{q'\in\Qcal} \beta_{qq'}^p u_{q'} \Big\}
  \label{eq:L2}
\end{align}
where $q\inc_{*}$ is the subset of $q\inc$ where the minimum is achieved in~\eqref{eq:L1}.
\end{proposition}
\begin{proof}
  Equations \eqref{eq:L1}-\eqref{eq:L2} are obtained by substituting $z(t)=\rho (t+t_0) + u$ in~\eqref{eq:petri_to_smdp}, letting $t$ tend to infinity and identifying slope and intercept for both sides, since $z$ is ultimately affine.

\medskip
  Conversely, suppose that~\eqref{eq:L1}-\eqref{eq:L2} hold, and consider
  $z(t)\coloneqq \rho(t+t_0)+u$. We need to show that for $t\geq 0$, $z$ satisfies the equation~\eqref{eq:petri_to_smdp}, that we may also rewrite under the form:
  \begin{align}
0= \min_{p\in q\inc} \Big\{ \Big(     c_q^p  + \kappa_q^p \sum_{q'\in \mathcal{Q}} \beta^p_{qq'} (u_{q'}
- \rho_{q'} \tau_p) -u_q \Big) + (t+t_0) \Big( \kappa_q^p \sum_{q'\in \mathcal{Q}} \beta^p_{qq'} \rho_{q'}- \rho_q \Big) \Big\}
\label{e-expand}
  \end{align}
  If $p\in q\inc_{*}$ achieves the minimum in~\eqref{eq:L2},
  so that it also achieves the minimum in~\eqref{eq:L1},
  then, the two
  terms in~\eqref{e-expand} vanish.
  Suppose now that $p$ achieves the minimum in~\eqref{eq:L1} but that
  it does not achieve the minimum in~\eqref{eq:L2}. Then, the coefficient
  of $(t+t_0)$ in~\eqref{e-expand} still vanishes,
  and, by~\eqref{eq:L2},
  \[
   c_q^p  +
       \kappa_q^p \sum_{q'\in \mathcal{Q}} \beta^p_{qq'} (u_{q'}
       - \rho_{q'} \tau_p) -u_q \geq 0  \enspace.
       \]
       Suppose finally that $p\in q\inc$ does not
       achieve the minimum in~\eqref{eq:L1}, which entails that the gap
       $\varepsilon\coloneqq\kappa_q^p\sum_{q'\in \Qcal} \beta_{qq'}^p\, \rho_{q'} - \rho_q$ takes a strictly positive value.
       Then, since $t\geq 0$,
       the expression in~\eqref{e-expand} can be
       bounded below by $C +
       t_0 \varepsilon  $ for some real constant $C$, and so, for $t_0$ large enough, this expression takes a nonnegative value, which entails that~\eqref{e-expand} holds.
\end{proof}

Observe how the equations~\eqref{eq:L1}-\eqref{eq:L2} derived for Petri nets asymptotic regimes are syntaxically the same than equations~\eqref{eq:OE1}-\eqref{eq:OE2} of Theorem~\ref{thm:ac}.
When a stoichiometric invariant $e$ exists, this can be seen as an immediate consequence of Correspondence Theorem~\ref{thm:correspondence2}. Indeed, the throughput $\rho_q$ of transition $q\in\Qcal$ is given by $\lim_{t\to\infty}z_q(t)/t$. Since, $z_q(t)/e_q$ corresponds to the value function of an undiscounted SMDP, the term $\rho_q/e_q$ is naturally interpreted as the optimal average cost of this SMDP starting from the state associated with $q$ (actually up to an inversion of limits in~\eqref{eq:avgcostdef} that Theorem~\ref{thm:stationary} thereafter proves licit).

Exploiting Correspondence Theorem~\ref{thm:correspondence2} further,
we arrive at our first main result, that provides existence of stationary regimes and uniqueness of the throughput.
\begin{theorem}\label{thm:stationary}
  Under Assumption~\ref{ass:A},
  \begin{enumerate}[label=(\roman*)]
  \item there exists an affine stationary regime, \ie~$(\rho,u)\in (\Rplus)^{\Qcal}\times \R^{\Qcal}$ such that, initializing the dynamics with $z(t)=\rho t+u$ for $t\in [-T,0]$, we end up with $z(t)=\rho t+u$ for all $t\geq 0$.
  \item the vector $\rho$ in (i) is universal, \ie, for any initial condition, the solution $z(t)$ of the dynamics
    satisfies
      \[ z(t)
      \underset{t\to\infty}{=} \rho t + O(1)\, .\]
  \end{enumerate}
\end{theorem}
\begin{proof}
  We begin by proving part (i). When the holding times of all places are unitary, the dynamics write $z(t)=F(z(t-1))$ where $F$ is piecewise affine and
  non-expansive in weighted sup-norm associated with $e$ from Propositions~\ref{prop:crandall}~and~\ref{prop:unitary}. Then, an affine stationary regime can be identified
  to an invariant half-line of $F$, whose existence follows from Kohlberg's Theorem~\ref{thm:kohlberg}. When the holding times are integer, and more generally, rational, we easily reduce to the unit delay case, exploiting Corollary~\ref{coro:reduc-to1}.
  However, when the holding times take irrational values, we cannot reduce
  to such a finite dimensional setting.
  From Proposition~\ref{prop-lex}, the existence of an affine stationary regime amounts to the existence of a solution to the lexicographic system~\eqref{eq:L1}-\eqref{eq:L2}, which from Correspondence Theorem~\ref{thm:correspondence2} is equivalent to the system~\eqref{eq:OE1}-\eqref{eq:OE2} of Theorem~\ref{thm:ac} on the average-cost of an undiscounted SMDP, as we explained above, and whose hypothesis is satisfied from our non-Zeno assumption.
  Denardo and Fox provided in~\cite{Denardo-Fox68} a constructive
  proof of the existence of the solution to this problem: a solution is obtained by applying a version of Howard's policy iteration algorithm adapted to multichain semi-Markov problems. The termination and correctness proofs in~\cite{Denardo-Fox68} entail the existence result.

  We now prove Assertion (ii): let $(\rho,u)$ be the stationary regime of assertion (i) and $\bar{z} \in \ZQcal$ the function such that for all $t\in[-T,0]$, $\bar{z}(t)=\rho t +u$.
  Consider another initial condition $\bar{z}' \in \ZQcal$,
  and let $z'$ denote the trajectory defined by this initial condition,
  so that $z'(t)=[\semigroup_t \bar{z}'](0)$.
  We know by Proposition~\ref{prop:nexp} that
  $\|\semigroup_t \bar z' - \semigroup_t \bar z\|_e
  \leq \|\bar{z}'-\bar z\|_e$. In particular,
  $|z'_q(t)-\rho_q t -u_q |/e_q \leq \|\bar z'-\bar z\|_e$, which proves the theorem.
\end{proof}

Now that we proved the existence of the throughput, we exploit the correspondences further to state three corollaries. First, Correspondence Theorem~\ref{thm:correspondence2} prompts us to introduce policies on Petri nets: a map $\sigma \colon \Qcal \to \Pcal$ is a \emph{policy} if for all $q\in\Qcal$, $\sigma(q)\in q\inc$. Given a policy $\sigma$ and a stoichiometric invariant $e$, the $|\Qcal|\times|\Qcal|$ matrix $P^{\sigma}$ with entries $(e_q^{-1}\kappa_q^{\sigma(q)}\beta^{\sigma(q)}_{qq'}e_{q'})_{q,q'\in\Qcal}$
is a probability matrix whose final classes are denoted by $\Fcal(\sigma)$. We denote by $\mu_F^{\sigma}$ the unique invariant measure supported by the class $F\in\Fcal$, and by $\phi^{\sigma}_{F,q}$ the probability of reaching $F$ by applying policy $\sigma$ starting from state $q$.
The vectors $m^{\sigma}$ (resp.\ $\tau^{\sigma}$) stand for $(m_{\sigma(q)})_{q\in \Qcal}$ and $(\tau_{\sigma(q)})_{q\in \Qcal}$ and we finally define the diagonal matrix $D^{\sigma} \coloneqq \mathrm{diag}((e_q^{-1}\alpha_{q\sigma(q)}^{-1}\pi_{q\sigma(q)})_{q\in\Qcal})$.  We then have the following result:

    \begin{corollary}[Throughput complex]
      \label{coro:complex}
      Under Assumption~\ref{ass:A}, the throughput vector $\rho$ is given by
   \begin{align}
     \forall q\in\Qcal, \quad \rho_q = e_q\,\min_{\sigma} \sum_{F\in\mathcal{F}(\sigma)} \phi_{F,q}^{\sigma} \frac{\langle \mu_F^{\sigma}, D^{\sigma}m^{\sigma}\rangle }{\langle \mu_F^{\sigma}, \tau^{\sigma} \rangle } \, ,
     \label{e-rhomin}
   \end{align}
   where the minimum is taken over all the policies.
    \end{corollary}
    \begin{proof}
      This is a consequence of Theorem~\ref{thm:ac} on undiscounted SMDPs.
    \end{proof}
    This formula shows that the throughput $\rho_q$ of the transition $q$ is a concave piecewise affine function of the initial marking vector $m\in (\Rplus)^{\Pcal}$. As is customary in tropical geometry, we associate with this map a polyhedral complex
(recall that a collection $\mathcal{L}$ of polyhedra is a {\em polyhedral complex} if for all $L\in\mathcal{L}$, any face $F$ of $L$ is also in $\mathcal{L}$ and for $L_1,L_2\in\mathcal{L}$, the polyhedron $L_1\cap L_2$ is a face of both $L_1$ and $L_2$, see~\cite{de2010triangulations}).
    If $\Sigma$ is a set of policies, we define the polyhedral {\em cell} $\mathcal{C}_\Sigma$
    to be the set of initial markings $m$ such that the argument of the minimum in~\eqref{e-rhomin} is $\Sigma$ (note that the cell
    $\mathcal{C}_\Sigma$ may be empty for some choices of $\Sigma$).
    The space $(\Rplus)^\Pcal$ is covered by the cells $\mathcal{C}_\Sigma$ of maximal
    dimension, the latter can be interpreted as {\em congestion phases}, or equivalently to a choice of bottleneck places for each $q\in\Qcal$ such that $|q\inc|>1$.

We now consider the computational complexity problem of computing the throughput vector $\rho$.
\begin{corollary}[LP characterization of the throughput]
  \label{coro:LP}
  Under Assumption~\ref{ass:A}, the throughput vector $\rho$ can be computed in polynomial time by solving
 the following linear program:

\begin{equation*}
    \max \sum_{q\in\Qcal} \rho_{q} \quad \text{s.t.}\quad
    \left\{\begin{aligned}
      \rho_q & \leq \kappa_q^p \sum_{q'\in \Qcal} \beta^p_{qq'}\rho_{q'}, &\quad \forall q\in \Qcal, \forall p\in q\inc \\
      u_q    &\leq c_q^p - \rho_q \tau_p + \kappa_q^p \sum_{q'\in \Qcal} \beta^p_{qq'}u_{q'}, &\quad \forall q\in \Qcal, \forall p\in q\inc
   \end{aligned}\right.
   \label{e-lp}
   \end{equation*}
in which $\rho, \, u\in\R^\Qcal$ are the variables. More precisely, if $(\rho, u)$ is any optimal solution of this program, then $\rho$ coincides with the throughput vector.
\end{corollary}
\begin{proof}
  This is an application of a theorem of Denardo and Fox~\cite[Th.~2]{Denardo-Fox68} on the undiscounted SMDP with value function $\tilde{z} = (z_q/e_q)_{q\in\Qcal}$. They indeed prove that for any positive vector $\nu\in\R^{\Qcal}$, the throughput vector $\tilde{\rho}$ is solution of the LP whose criterion is $\sum_{q\in\Qcal} \nu_k \tilde{\rho}_q$ and whose feasibility set is defined by inequalities $\tilde{\rho}_q \leq \sum_{q'\in\Qcal}e_q^{-1}\kappa^p_q\beta^p_{qq'}e_{q'}\tilde{\rho}_{q'}$ and $\tilde{u}_q \leq e_q^{-1}c^p_q - \tilde{\rho}_q\tau_p + \sum_{q'\in\Qcal}e_q^{-1}\kappa^p_q\beta^p_{qq'}e_{q'}\tilde{u}_{q'}$ for all $q\in\Qcal$ and $p\in q\inc$. Choosing $\nu = e$ and switching back to variables $(\rho,u)$ gives the announced result.
  Eventually, recall that linear programs can be solved in (weak) polynomial time
  by the ellipsoid or by interior point methods.
\end{proof}

In their work~\cite{gaujal2004optimal}, Gaujal and Giua developed a closely related linear programming approach, derived directly from Little's law, rather than from the theory of semi-Markov processes. Their formulation has same objective function and a feasibility set that only differs from the one of Corollary~\ref{coro:LP} for transitions in $\Qpsel$. However, by applying the conflict-free transformation that they suggest and that we have introduced in our last remark after~Correspondence Theorem~\ref{thm:correspondence2}, we can
recover the formulation of~\cite{gaujal2004optimal} from Corollary~\ref{coro:LP}.

\medskip
The asymptotic behavior of the value function in large horizon has been extensively studied~\cite{schweitzer79,spectral}. As a corollary of these results, we
arrive at:
\begin{corollary}[Asymptotic Periodicity]\label{coro:periodic}
  Suppose that Assumption~\ref{ass:A} holds, and that the holding times are integer (so that $T \in \N$). Then, there exists an integer   $c$, which is the order of an element of $\mathrm{Sym}(nT)$, such that, for all $0\leq r\leq c-1$,
  $z(tc+r)-\rho(tc+r)$ converges as $t\to\infty$, for integer values of $t$.
\end{corollary}
Whereas the earlier results of this section
hold for irrational holding times, the integrality
restriction in Corollary~\ref{coro:periodic} is essential.
\begin{proof}
  We use Corollary~\ref{coro:reduc-to1} to reduce to a system
  of the form $\tilde{z}(t) = F(\tilde{z}(t-1))$, where
  $\tilde{z}(t)$ is the augmented vector $(z(t),\dots,z(t-T+1))\in \R^{\Qcal\times \{1,\dots,T\}}$,
  and $F$ is order-preserving, additively homogenous with respect to $e$, piecewise affine and concave.
    Then, the result follows from~Theorem~\ref{thm:schweitzer}.
\end{proof}

\subsection{Application to model~\eqref{SAMU1}}
\label{sec:correspondencecasestudy}

We illustrate the above results on our running example~\eqref{SAMU1}. Since $z_1$ and $z_3$ have two upstream places each, there are a total of four policies. Though it is possible to use~\eqref{e-rhomin} to determine $\rho$, solving the lexicographic equations~\eqref{eq:L1}-\eqref{eq:L2} turns out to be easier in practice. We also remark (for instance on~\Cref{fig:SMDP}) that $z_2$ (resp.\ $z_4$ and $z_5$) are always in the same recurrence class as $z_1$ (resp.\ $z_3$), in the~sense of $\,$SMDP's$\,$ chains. As a result, we shall just focus on the lexicographic optimality equation %on
\eject
\noindent  on $\rho_1$ and $\rho_3$:
\[
\left\{\begin{array}{rcl}
  (\rho_1, u_1) & = & (\lambda ,\, 0)  \wedge  ((1-\pi)\rho_1+\rho_3,\, N_A+(1-\pi)(u_1-\rho_1\tau_1)+u_3-\rho_3\tau_2)\\[1ex]
  (\rho_3, u_3) & = & \pi(\rho_1 ,\, u_1-\rho_1\tau_1)  \wedge  (\rho_3,\, N_P + u_3 -\rho_3(\tau_2+\tau_3)\big)\\
\end{array}\right.
\]
where $\wedge$ now stands for the $\min^{\textsc{lex}}$ operation.
Each policy (\ie~each choice of minimizing term in both equations) leads to a value of $\rho_1$ and $\rho_3$ and provides linear inequalities characterizing the associated validity domain. Eventually, we obtain $\rho_1=\rho^*$ and $\rho_3=\pi\rho^*$ with
\[\rho^*=\min\left(\lambda\,,\,\frac{N_A}{\tau_1+\pi\tau_2}\,,\,\frac{N_P}{\pi(\tau_2+\tau_3)}\right) \]

\noindent in which we retrieve the piecewise-affine form of $\rho$ showed in Corollary~\ref{coro:complex}.

  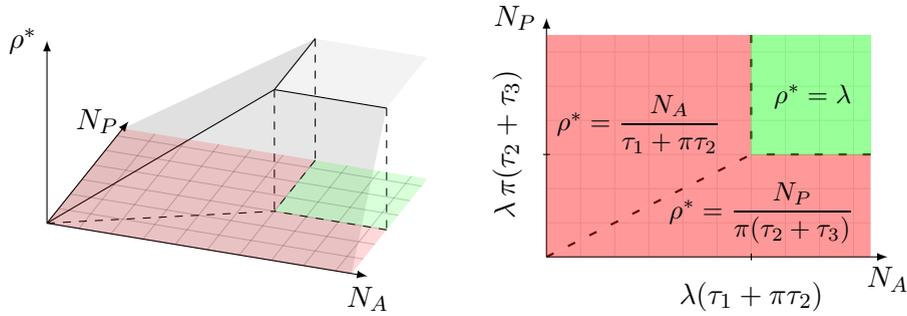
\begin{figure}[h]
   % \vspace{-2mm}
    \centering
    \begin{tabular}{lr}
    \def\tkzscl{.5}  % ./phases_diagram_1_3D.tex

\tdplotsetmaincoords{63}{20}
\hskip-2ex
\begin{tikzpicture}[tdplot_main_coords,scale=\tkzscl]

    \begin{scope}[scale=0.9]

  \draw[->,>=latex] (0,0,0) -- coordinate (x axis mid) (10,0,0) node[below] {$N_A$};
  \draw[->,>=latex] (0,0,0) -- coordinate (y axis mid) (0,7,0) node[left] {$N_P$};
  \draw[->,>=latex] (0,0,0) -- (0,0,6) node[left] {$\rho^*$};

  \foreach \x in {0,...,9}
    \draw[thin, opacity=.2] (\x,0)--(\x,6.5) ;
  \foreach \y in {0,...,6}
    \draw[thin, opacity=.2] (0,\y)--(9.5,\y) ;

  \coordinate (x0) at (0,0);
  \coordinate (x1) at (9.5,3,4);
  \coordinate (x2) at (6,6.5,4);
  \coordinate (x10) at (9.5,3,0);
  \coordinate (x20) at (6,6.5,0);
  \coordinate (x3) at (6,3,4);
  \coordinate (x30) at (6,3,0);
  \coordinate (xh0) at (6,0);
  \coordinate (xh1) at (6,6.5,4);
  \coordinate (xh4) at (9.5,6.5,4);
  \coordinate (xh10) at (6,6.5,0);
  \coordinate (xh40) at (9.5,6.5,0);
  \coordinate (xh2) at (0,3);
  \coordinate (xh3) at (9.5,3);
  \coordinate (xh5) at (9.5,0);
  \coordinate (xh6) at (0, 6.5);

 \draw (x0) -- (x3);
 \draw (x1) -- (x3);
 \draw (x2) -- (x3);

 \draw[dashed] (x0)  -- (x30);
 \draw[dashed] (x10) -- (x30);
 \draw[dashed] (x20) -- (x30);
 \draw[dashed] (x3)  -- (x30);
 \draw[dashed] (x1)  -- (x10);
 \draw[dashed] (xh1)  -- (xh10);

  \fill[black, opacity=0.125] (x0) -- (x3) -- (xh1) -- (xh6) -- cycle;
  \fill[black, opacity=0.05] (x3) -- (x1) -- (xh4) -- (x2) -- cycle;
  \fill[black, opacity=0.1] (x0) -- (x3) -- (x1) -- (xh5) -- cycle;

  \fill[red, opacity=0.15] (x0) -- (x30) -- (xh10) -- (xh6) -- cycle;
  \fill[green,opacity=0.15] (x30) -- (x10) -- (xh40) -- (x20) -- cycle;
  \fill[red,opacity=0.15] (x0) -- (x30) -- (x10) -- (xh5) -- cycle;

\end{scope}

\end{tikzpicture}
 &
    \def\tkzscl{.5}  % ./phases_diagram_1.tex

\begin{tikzpicture}[scale=\tkzscl]
\begin{scope}[scale=0.9]
\draw[step=1,black,thin,opacity = 0.1] (0,0) grid (9.5,6.5);
  \draw[->,>=latex] (0,0) -- coordinate (x axis mid) (10,0) node[below] {$N_A$};
  \draw[->,>=latex] (0,0) -- coordinate (y axis mid) (0,7) node[left] {$N_P$};

  \coordinate (x0) at (0,0);
  \coordinate (x1) at (9.5,3);
  \coordinate (x2) at (6,6.5);
  \coordinate (x3) at (6,3);
  \coordinate (xh0) at (6,0);
  \coordinate (xh1) at (6,6.5);
  \coordinate (xh2) at (0,3);
  \coordinate (xh3) at (9.5,3);
  \coordinate (xh4) at (9.5,6.5);
  \coordinate (xh5) at (9.5,0);
  \coordinate (xh6) at (0, 6.5);

 \draw[thick, loosely dashed] (x0) -- (x3);
 \draw[thick, loosely dashed] (x1) -- (x3);
 \draw[thick, loosely dashed] (x2) -- (x3);

  \fill[red, opacity=0.4] (x0) -- (x3) -- (xh1) -- (xh6) -- cycle;
  \fill[green,opacity=0.4] (x3) -- (x1) -- (xh4) -- (x2) -- cycle;
  \fill[red,opacity=0.4] (x0) -- (x3) -- (x1) -- (xh5) -- cycle;
\node[align = center] at (2.7,3.9) {\small$\rho^* = \displaystyle\frac{N_A}{\tau_1+\pi\tau_2}$};
\node[align = center] at (7.8,4.7) {\small$\rho^* = \lambda$};
\node[align = center] at (6.3,1.3) {\small$\rho^* = \displaystyle\frac{N_P}{\pi(\tau_2+\tau_3)}$};
   \draw (6,-.1) -- (6,.1) node[color=black, label={[color=black,label distance=.1cm]below:$\lambda(\tau_1+\pi\tau_2)$}] {};
  \draw (-.1,3) -- (.1,3) node[color=black,anchor=east,label={[color=black,rotate=90, label distance=.1cm]$\lambda\,\pi(\tau_2+\tau_3)$}] {};
\end{scope}
\end{tikzpicture}
    \end{tabular}
    \vspace{-2mm}
    \caption{The phase diagram of the~\eqref{SAMU1} system}\label{fig:phasediagram}%\vspace{-.5cm}
\end{figure}

We interpret this result as follows: the ``handling speed'' $\rho_1$ of the MRAs and $\rho_3$ of the emergency physicians are always entangled and depend on three key dimensioning parameters: the arrival rate of inbound calls $\lambda$, the maximum MRA throughput $N_A/(\tau_1+\pi\tau_2)$ and the maximum physician throughput $N_P/\pi(\tau_2+\tau_3)$. We recognize in these last two terms a number of agents divided by a characteristic cycle time. Hence, if $N_A\geq N_A^*\coloneqq \lambda(\tau_1+\pi\tau_2)$ and $N_P\geq N_P^*\coloneqq\lambda\pi(\tau_2+\tau_3)$ (this delineates the green phase on~\Cref{fig:phasediagram}), we have $\rho^*=\lambda$ which means that all inbound calls are handled. If $N_A\leq N_A^*$ and $N_A/(\tau_1+\pi\tau_2)\leq N_P/\pi(\tau_2+\tau_3)$, there are too few MRAs, therefore they impose their maximum handling speed to the whole system (indeed emergency physicians wait for MRAs to pass them calls). Conversely, if $N_P\leq N_P^*$ and $N_P/\pi(\tau_2+\tau_3)\leq N_A/(\tau_1+\pi\tau_2)$, there are too few emergency physicians and they impose their handling speed to the whole system again (MRAs are waiting for doctors to take their calls and be released). This is illustrated by the phase diagram depicted on \Cref{fig:phasediagram}: a staffing choice $(N_A,N_P)$ made in real-life call-enter will indicate the long-run throughput $\rho^*$ of the system. We verify that the cells of the phase diagram are the regions over which $\rho^*$ is affine (as a function of $N_A$ and $N_P$).

    To sum it up, there are three different possible regimes, among which only one is fluid and guarantees that all calls are answered. This analysis can lead to minimal dimensioning recommendations: for such an emergency call center and considering that calls arrive with rate $\lambda$, at least $\lceil\lambda(\tau_1+\pi\tau_2)\rceil$ MRAs and $\lceil\lambda\pi(\tau_2+\tau_3)\rceil$ emergency physicians are needed.

\section{Stationary regimes in presence of priorities}

\subsection{From stationary regimes to systems on germs}

There is a convenient and more abstract way to write the lexicographic system \eqref{eq:L1}-\eqref{eq:L2} of Proposition~\ref{prop-lex}, that we somehow already used in~\Cref{sec:correspondencecasestudy}, using germs of affine functions. A {\em germ} at infinity of a function $f$ is an equivalence class for the relation which identifies two functions that coincide for sufficiently large values of their argument. The tuple $(\rho,u)\in\R^2$ will represent the germ of the affine function $t \mapsto \rho t +u$. The pointwise order on functions induces a total order on germs of affine
functions, which coincides with the lexicographic order on the coordinates
$(\rho,u)$, the $\rho$ coordinate being considered first.
We complete $\R^2$ by introducing a greatest element $\top$
with respect to the lexicographic minimum.
Then, $\Germ:=\R^2\cup\{\top\}$ equipped with the operations $\min^{\textsc{lex}}$ and $+$ is a semifield (by convention, for all $g\in\Germ$, $g+\top=\top+g=\top$).
The multiplicative group $\R_{>0}$ acts
on $\Germ$ by setting $a(\rho,u):=(a\rho,au)$,
for $a>0$ and $(\rho,u)\in \R^2$, and $a\top = \top$. If $(\rho, u)$ is the germ of $f$, it is immediate to see that $(\rho, u - \rho\tau)$ is the germ of $t\mapsto f(t-\tau)$.

\medskip
In this framework, the system~\eqref{eq:L1}-\eqref{eq:L2} derived from~\eqref{eq:petri_to_smdp} becomes:
\begin{equation}
  \label{eq:L}
  \forall q\in\Qcal \quad (\rho_q,u_q) = \underset{p\in q\inc}{\min{}^{\textsc{lex}}}\Big((0, c^p_q) + \kappa_q^p\sum_{q'\in \Qcal}\beta_{qq'}^p(\rho_{q'},u_{q'}-\rho_{q'}\tau_p)\Big)
      \tag{L}
\end{equation}

If there are some priority routings, Correspondence Theorems~\ref{thm:correspondence1} and~\ref{thm:correspondence2} do not hold anymore: the dynamics has still the form of a Bellman equation,
but the factors $\beta_{qq'}^p$ in~\eqref{eq:petri_to_smdp} take negative values,
implying that some ``probabilities'' are negative. However, it is still relevant
to look for affine stationary regimes, and we next show that these
regimes are the solutions of a lexicographic system over germs similar to~\eqref{eq:L}.
To do so, we derive other germ equations for transitions ruled by priority routing, whose dynamics is recalled in~\Cref{table:equation_counters}.
In particular, one needs to address how the expressions of the form $z(t^-)$ behave when passing to germs. The problem may seem ill-posed since this value coincides with $z(t)$ for ultimately affine functions. Nonetheless, in~\cite{allamigeon2015performance}, it has been shown that the problem of looking for ultimately affine stationary regimes on the $\delta$-discretization of the fluid dynamics is well-posed. In this discretized model, the term $z(t^-)$ is replaced by $z(t- \delta)$. The detour \textit{via} this discretized dynamics enables one to prove that, regardless of the choice of $\delta$, small enough, some terms cannot achieve the minimum in the priority dynamic equations, and thus can be removed. This leads to the last equation of~\Cref{table:germs_equation_counters} and the following result.

\definecolor{gray1}{gray}{0.82}
\definecolor{gray2}{gray}{0.94}
\setlength{\arrayrulewidth}{.8pt}
\begin{table}[ht]
\vspace*{-2mm}
  \caption{Dynamic equations followed by germs of transitions counter functions}
  \label{table:germs_equation_counters}\vspace{-1mm}
  \makebox[\textwidth][c]{
 \scalebox{0.96}{
     \begin{tabular}{C{1.6cm}!{\color{white}{\vrule width .8pt}}S{L{13.7cm}}}
      \rowcolor{gray1} Type  & \multicolumn{1}{c}{Germs equation in stationary regime}  \\
      \arrayrulecolor{white}\hline
      \rowcolor{gray2} \(q\in\Qsync\)  &\hspace*{-1mm}  \(\displaystyle \, (\rho_q,u_q)= \hskip.2ex \underset{p\in q\inc}{{\min{}^{\textsc{lex}}}} \hskip.6ex \alpha_{qp}^{-1}\bigg((0,m_p) + \!\! \sum_{q'\in p\inc} \alpha_{pq'} \, (\rho_{q'},u_{q'}-\rho_{q'}\tau_p)\bigg)\)  \\
      \arrayrulecolor{white}\hline
      \rowcolor{gray2} \(q\in\Qpsel\)  &\hspace*{-1mm} \(\displaystyle \, (\rho_q,u_q)= \hskip2.4ex \pi_{qp} \hskip.2ex \cdot \hskip.8ex  \alpha_{qp}^{-1}\bigg((0,m_p) + \!\! \sum_{q'\in p\inc} \alpha_{pq'} \, (\rho_{q'},u_{q'}-\rho_{q'}\tau_p)\bigg) \)  \\
      \arrayrulecolor{white}\hline
      \rowcolor{gray2} \(q\in\Qprio\)  &\hspace*{-1mm}  \(\displaystyle \, (\rho_q,u_q)\!=\! \underset{\substack{p\in q\inc\\[1ex]\underset{q'\succ_p q}{\sum}\rho_{q'}=0}}{\min{}^{\textsc{lex}}} \alpha_{qp}^{-1}\bigg((0,m_p) + \!\! \sum_{q'\in p\inc}\! \alpha_{pq'} \, (\rho_{q'},u_{q'}-\rho_{q'}\tau_p)-\!\!\!\!\sum_{q' \in p\out\setminus\{q\}}\!\alpha_{q'p}(\rho_{q'},u_{q'})\bigg)\) \\
      \arrayrulecolor{white}\hline
    \end{tabular} }
  }
  \vspace{-2mm}
\end{table}

\begin{theorem}\label{th:germ_prio}
The ultimately affine stationary regimes $z(t) = \rho t + u$ of the dynamics of~\Cref{table:equation_counters} are solutions of the germ equations of~\Cref{table:germs_equation_counters}.
\end{theorem}
\begin{proof}
  We prove the result for transitions ruled by priority, since the result is direct for other patterns.
  Recall from~\Cref{table:equation_counters} the counter equation followed by transition $q\in\Qprio$:
  \[
    z_q(t)= \min_{p\in q\inc} \alpha_{qp}^{-1}\bigg(m_p + \!\! \sum_{q'\in p\inc} \alpha_{pq'} \, z_{q'}(t-\tau_p)-\!\!\sum_{q' \mathrel{\prec_p} q}\!\alpha_{q'p}z_{q'}(t)-\!\!\sum_{q' \mathrel{\succ_p} q}\!\alpha_{q'p}z_{q'}(t^-)\bigg)
    \]
    Our claim is that due to the priority mechanism, some terms in the above minimum cannot realize minimality and thus can be removed.

\medskip
    Let $q_1\in\Qprio$, $p\in q\inc$ and $q_2\in p\out$, with $q_2\neq q_1$. Let us suppose that $q_1\prec_p q_2$.
    Substituting counters by their corresponding germs and
    replacing $t^-$ by $t-\delta$ with $\delta>0$, we have
    \begin{align*}
      (\rho_{q_1},u_{q_1}) &\leq \alpha_{q_1 p}^{-1}\bigg((\rho^p_{\Sigma},u^p_{\Sigma})-\!\!\sum_{q' \prec_p\, q_1}\!\alpha_{q'p}(\rho_{q'},u_{q'}) -\!\!\sum_{q' \succ_p\, q_1}\!\alpha_{q'p}(\rho_{q'},u_{q'}-\rho_{q'}\delta) \bigg) \\
           &\leq \alpha_{q_1 p}^{-1}\bigg((\rho^p_{\Sigma},u^p_{\Sigma})-\!\!\sum_{q' \neq\, q_1}\!\alpha_{q'p}(\rho_{q'},u_{q'}) +\!\!\sum_{q' \succ_p\, q_1}\!\alpha_{q'p}(0,\rho_{q'}\delta) \bigg) \label{germU1}\tag{U1}\\
    \end{align*}
      \vspace{-1cm}

\noindent  where $(\rho^p_{\Sigma},u^p_{\Sigma})$ stands for the germ $\displaystyle(0,m_p) + \!\! \sum_{q'\in p\inc} \alpha_{pq'} \, (\rho_{q'},u_{q'}-\rho_{q'}\tau_p)$.
Similarly, we have
\begin{align*}
  (\rho_{q_2},u_{q_2}) &\leq \alpha_{q_2 p}^{-1}\bigg((\rho^p_{\Sigma},u^p_{\Sigma})-\!\!\sum_{q' \neq\, q_2}\!\alpha_{q'p}(\rho_{q'},u_{q'}) +\!\!\sum_{q' \succ_p\, q_2}\!\alpha_{q'p}(0,\rho_{q'}\delta) \bigg) \label{germU2}\tag{U2}\\
\end{align*}
\vspace{-1cm}

In both sides of the latter equation, let us apply the nondecreasing
mapping of $\Germ\to\Germ$ : $g\mapsto \alpha_{q_2p}\alpha_{q_1p}^{-1}\big(g - (\rho_{q_2},u_{q_2})\big)+(\rho_{q_1},u_{q_1})$. We obtain:

\[
  (\rho_{q_1},u_{q_1}) \leq \alpha_{q_1 p}^{-1}\bigg((\rho^p_{\Sigma},u^p_{\Sigma})-\!\!\sum_{q' \neq\, q_1}\!\alpha_{q'p}(\rho_{q'},u_{q'}) +\!\!\sum_{q' \succ_p\, q_2}\!\alpha_{q'p}(0,\rho_{q'}\delta) \bigg) \label{germU2p}\tag{U2'}\\
\]

Comparing~\eqref{germU2p} to~\eqref{germU1}, one can observe that if $\rho_{q'}>0$ for some $q_1\prec_p q'\preceq_p q_2$, then the right-hand side of~\eqref{germU2p} strictly bounds by below the one of~\eqref{germU1}, thus the equality in~\eqref{germU1} cannot be achieved, and the corresponding germ can be removed from the original minimum. This reasoning can be applied when $q_2$ is the transition of $p\out$ with the least priority, so that the inequality~\eqref{germU1} is strict whenever the sum $\sum_{q'\succ_p q_1}\rho_{q'}$ is positive (the $(\rho_q)_{q\in\Qcal}$ variables are nonnegative). Conversely, only the contributions of upstream places $p\in q_1\inc$ such that $\sum_{q'\succ_p q_1}\rho_{q'}=0$ remain in the minimum.
\end{proof}

\subsection{Case Study: application to the model~\eqref{eq:SAMU2}}
\label{sec:casestudyprio}
In this section, we follow up on the analysis of the model~\eqref{eq:SAMU2} involving priority rules. Using \Cref{table:equation_counters}, we can write the dynamics of the counter variables of the net. We present below a reduced system of equations where $z_2$, $z_4$, $z_6$, $z_6'$, $z_7$ and $z_7'$ have been substituted by expressions depending on $z_1$, $z_3$, $z_5$ and $z_5'$ only. For the sake of readability, we denote $z|^{t_2}_{t_1} \coloneqq z(t_2)-z(t_1)$, and $z|^t \coloneqq z(t)$.

\resizebox{.84\textwidth}{!}{
  \begin{minipage}{\textwidth}
      \begin{equation*}\label{eq:SAMU2}\tag{EMS-B}
      \hspace{-.75cm}\begin{array}{rcrclcl}
      z_1(t) &=& z_0\big|^t \;\hskip2.8ex\; &\wedge& \left(N_A+(1-\pi)\,z_1\big|^{t-\tau_1}+z_3\big|^{t-\tau_2}\right) & & \\[2ex]
      z_3(t) &=& \pi\,z_1\big|^{t-\tau_1}& & &\wedge&\left(N_R+z_3\big|^{t-\tau_2}+z_5\big|_{t}^{t-\tau_2}+z_5'\big|_{t}^{t-\tau_2}\right)\\[2ex]
      z_5(t) &=& \alpha z_3\big|^{t-\tau_2}&\wedge&\left(N_P+z_5\big|^{t-\tau_2-\tau_3}+z_5'\big|^{t-\tau_2-\tau_3}_{t^{-}}\right)&\wedge&\left(N_R   +z_3\big|^{t-\tau_2}_ {t^{-}}+z_5\big|^{t-\tau_2}+z_5'\big|^{t-\tau_2}_{t^{-}}\right)\\[2ex]
      z_5'(t)&=& (1-\alpha) z_3\big|^{t-\tau_2}&\wedge&\left(N_P+z_5\big|^{t-\tau_2-\tau_3}_t+z_5'\big|^{t-\tau_2-\tau_3}\right)&\wedge&\left(N_R    +z_3\big|^{t-\tau_2}_ {t^-}+z_5\big|^{t-\tau_2}_ t+z_5'\big|^{t-\tau_2}\right)\\[2ex]
      \end{array}
    \end{equation*}
  \end{minipage}
}

\vspace{.3cm}

Applying Theorem~\ref{th:germ_prio} and the equations of~\Cref{table:germs_equation_counters} to the model~\eqref{eq:SAMU2} provide the following system on the affine germs of counter variables (again after substitutions of some germs easily expressed in terms of those assiociated with counters $z_1$, $z_3$, $z_5$ and $z_5'$):

\resizebox{.83\textwidth}{!}{
  \begin{minipage}{\textwidth}
      \begin{equation*}\label{priogerm}
        \hspace{-.9cm}\begin{aligned}
  (\rho_1,u_1) &= (\lambda,0) \wedge ((1-\pi)\rho_1+\rho_3,N_A+(1-\pi)(u_1-\rho_1\tau_1)+u_3-\rho_3\tau_2)  \\[1ex]
  (\rho_3,u_3) &= (\pi\rho_1,\pi(u_1-\rho_1\tau_1)) \wedge(\rho_3,u_3+N_R-\rho_3\tau_2-(\rho_5+\rho_5')\tau_2) \\[1ex]
  (\rho_5,u_5) &= \left\{\begin{aligned}\alpha(\rho_3,u_3-\rho_3\tau_2)\wedge(\rho_5,u_5+N_P-\rho_5(\tau_2+\tau_3))\wedge(\rho_3,u_3+N_R-(\rho_3+\rho_5)\tau_2)  & \quad \text{if}\quad \rho_5'=0 \;\; \text{and}\;\;\rho_3 = 0  &\\ \alpha(\rho_3,u_3-\rho_3\tau_2)\wedge(\rho_5,u_5+N_P-\rho_5(\tau_2+\tau_3))& \quad \text{if}\quad \rho_5'=0 \;\; \text{and}\;\;\rho_3 > 0  &\\ \alpha(\rho_3,u_3-\rho_3\tau_2) & \quad \text{if}\quad \rho_5'>0 \end{aligned}\right. \\[1ex]
  (\rho_5',u_5') &= \left\{\begin{aligned}(1-\alpha)(\rho_3,u_3-\rho_3\tau_2)\wedge(\rho_5',u_5'+N_P-(\rho_5+\rho_5')(\tau_2+\tau_3))\wedge(\rho_5',N_R+u_5'-(\rho_5+\rho_5')\tau_2)  & \quad \text{if}\quad \rho_3=0  &\\(1-\alpha)(\rho_3,u_3-\rho_3\tau_2)\wedge(\rho_5',u_5'+N_P-(\rho_5+\rho_5')(\tau_2+\tau_3)) & \quad \text{if}\quad \rho_3>0 \end{aligned}\right. \\[1ex]
\end{aligned}
\end{equation*}
\end{minipage}
}

\medskip
The major difference of this system on germs compared with the one obtained for the model \eqref{SAMU1} in~\Cref{sec:correspondencecasestudy} is the necessity, brought by priorities, to distinguish cases on the possible bottleneck upstream places depending on their respective throughputs. We point out that the cases where $\rho_3=0$ in germs equations on $(\rho_5,u_5)$ and $(\rho_5',u_5')$ could acceptably be neglected for further analysis. Indeed, the first two equations above always ensure $\rho_3=\pi\rho_1$, and supposing $\rho_3=0$ and $\lambda>0$ leads to $\min(N_A,N_R)=0$ by combination of the attained germs. Therefore, the throughput $\rho_3$ is positive as soon as we suppose $\lambda>0$, $N_A>0$ and $N_R>0$, \ie, positive inflow of calls and positive number of agents to pick them up.
As a result, when this condition is met, the priority ruling the routing of tokens from the pool of reservoir assistants does not appear on the affine germs of $z_5$ and $z_5'$ anylonger. This is an expected outcome since transitions $z_5$ and $z_5'$ (high level of priority for the reservoir pool) can only receive tokens that have passed through transition $z_3$ (low level of priority for the reservoir pool) before, as a result $z_5$ and $z_5'$ cannot ultimately inhibit themselves. Such a layout of priorities does remain appropriate to perform arbitration of tokens orientation in case of conflicts and we show below that it still produces effects in the scope of long-run time analysis of the system.

As in~\Cref{sec:correspondencecasestudy}, a choice of policy (\ie~a choice of minimizing terms in the lexicographic system) provides affine equalities determining the throughput as an affine function of the resources of the model, the validity region of this expression being obtained by inequalities derived from the remaining (non-minimizing) terms of the system. This leads to nine full-dimensional congestion phases (maximal cells of the throughput complex) covering $(\Rplus)^3$, that we depict and number on~\Cref{fig:phasesRes} and whose polyhedral form is given in~\Cref{ap:congestionphasesdetailed} along with the expressions of throughputs.

As expected, the introduction of a new type of resource agent (the reservoir assistant) introduces more slowdown phases if its initial marking $N_R$ is too small. Therefore, to ensure the good behavior of the~\eqref{eq:SAMU2} model whose design relies substantially on the reservoir, one needs to take $N_R\geq N_R^*\coloneqq 2\pi\lambda\tau_2$. Note that the minimal number of MRAs (resp.\ emergency physicians) to answer all the calls is not affected by the presence of the reservoir by comparison with~\eqref{SAMU1} model, and is still equal to $N_A^* \coloneqq \lambda(\tau_1+\pi\tau_2)$ (resp.\ $N_P^*\coloneqq \pi\lambda(\tau_2+\tau_3)$). These three lower bounds on $N_A$, $N_R$ and $N_P$ define the phase \phase{1}, that we may refer to as the ``fluid phase''.

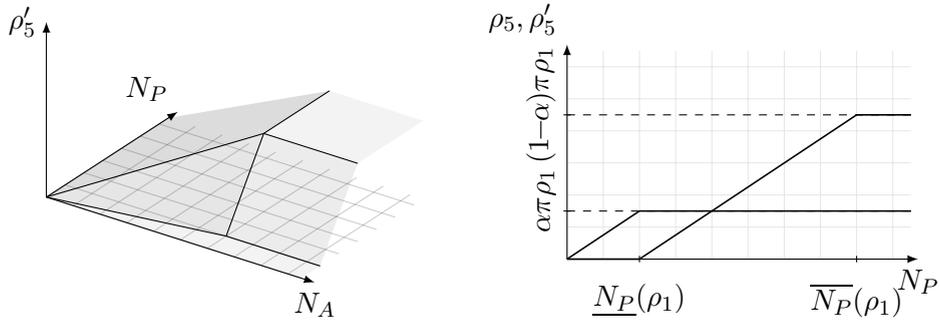
\begin{figure}[h]
  \hspace{-.3cm}
  \centering
  \begin{tabular}{lr}
  \def\tkzscl{.48}  % ./phases_diagram_2_3D.tex

\tdplotsetmaincoords{63}{35}
\hskip-2ex
\begin{tikzpicture}[tdplot_main_coords,scale=\tkzscl]

    \begin{scope}[scale=0.9]

  \draw[->,>=latex] (0,0,0) -- coordinate (x axis mid) (10,0,0) node[below] {$N_A$};
  \draw[->,>=latex] (0,0,0) -- coordinate (y axis mid) (0,7,0) node[above left] {$N_P$};
  \draw[->,>=latex] (0,0,0) -- (0,0,6) node[left] {$\rho_5'$};

  \def\alp{0.33}
  \def\bet{0.67}

  \foreach \x in {0,...,9}
    \draw[thin, opacity=.2] (\x,0)--(\x,6.5) ;
  \foreach \y in {0,...,6}
    \draw[thin, opacity=.2] (0,\y)--(9.5,\y) ;

  \coordinate (x0) at (0,0);
  \coordinate (x1) at (9.5,1,4*\alp);
  \coordinate (x2) at (6,6.5,4*\alp);
  \coordinate (x1n) at (9.5,3,4*\bet);
  \coordinate (x2n) at (6,6.5,4*\bet);
  \coordinate (x10) at (9.5,1,0);
  \coordinate (x10n) at (9.5,3,0);
  \coordinate (x20) at (6,6.5,0);
  \coordinate (x3) at (6,1,4*\alp);
  \coordinate (x3n) at (6,3,4*\bet);
  \coordinate (x30) at (6,1,0);
  \coordinate (x30n) at (6,3,0);
  \coordinate (xh0) at (6,0);
  \coordinate (xh1) at (6,6.5,4*\alp);
  \coordinate (xh4) at (9.5,6.5,4*\alp);
  \coordinate (xh1n) at (6,6.5,4*\bet);
  \coordinate (xh4n) at (9.5,6.5,4*\bet);
  \coordinate (xh10) at (6,6.5,0);
  \coordinate (xh40) at (9.5,6.5,0);
  \coordinate (xh2) at (0,3);
  \coordinate (xh3) at (9.5,3);
  \coordinate (xh5) at (9.5,0);
  \coordinate (xh6) at (0, 6.5);

 \draw (x0) -- (x3n);
 \draw (x1n) -- (x3n);
 \draw (x2n) -- (x3n);
 \draw (x30) -- (x3n);

 \draw (x0)  -- (x30);
 \draw (x10) -- (x30);

  \fill[black, opacity=0.14] (x0) -- (x3n) -- (xh1n) -- (xh6) -- cycle;
  \fill[black, opacity=0.05] (x3n) -- (x1n) -- (xh4n) -- (x2n) -- cycle;
  \fill[black, opacity=0.1] (x0) -- (x3n) -- (x30) -- cycle;
  \fill[black, opacity=0.075] (x30) -- (x3n) -- (x1n) -- (x10) -- cycle;
  \fill[black, opacity=0.05] (x30) -- (x0) -- (xh5) -- (x10) -- cycle;

\end{scope}

\end{tikzpicture}
 &
  \def\tkzscl{.48}  % ./phases_diagram_2.tex

\begin{tikzpicture}[scale=\tkzscl]
\begin{scope}[scale=0.88,xscale=1.13]
\draw[step=1,black,thin,opacity = 0.1] (0,0) grid (9.5,6.5);
  \draw[->,>=latex] (0,0) -- coordinate (x axis mid) (9.7,0) node[below] {$N_P$};
  \draw[->,>=latex] (0,0) -- coordinate (y axis mid) (0,6.7) node[above left] {$\rho_5,\rho_5'$};

  \coordinate (x0) at (0,0);
  \coordinate (x1) at (9.5,1.5);
  \coordinate (x2) at (2,1.5);
  \coordinate (x3) at (2,0);
  \coordinate (x4) at (8,4.5);
  \coordinate (x5) at (9.5,4.5);
  \coordinate (x6) at (0,1.5);
  \coordinate (x7) at (0,4.5);

  \draw (8,-.1) -- (8,.1) node[color=black, label={[color=black,label distance=.1cm]below:$\overline{N_P}(\rho_1)$}] {};
  \draw (2,-.1) -- (2,.1) node[color=black, label={[color=black,label distance=.1cm]below:$\underline{N_P}(\rho_1)$}] {};
  \draw (-.1,1.5) -- (.1,1.5) node[color=black,anchor=west,label={[color=black, label distance=.2cm,rotate=90]:$\alpha\pi\rho_1$}] {};
  \draw (-.1,4.5) -- (.1,4.5) node[color=black,anchor=west,label={[color=black, label distance=.2cm,rotate=90]:$(1\!\!-\!\!\alpha)\pi\rho_1$}] {};

    \draw[semithick] (x0) -- (x2) -- (x1);
    \draw[semithick] (x0) -- (x3) -- (x4) -- (x5);
    \draw[dashed] (x6) -- (x1);
    \draw[dashed] (x7) -- (x5);
    \fill[green,opacity=0.] (8,0) -- (8,6.5) -- (9.5,6.5) -- (9.5,0) -- cycle;
    \fill[orange,opacity=0.] (2,0) -- (2,6.5) -- (0,6.5) -- (0,0) -- cycle;

    \fill[yellow,opacity=0.] (2,0) -- (2,6.5) -- (8,6.5) -- (8,.0) -- cycle;
\end{scope}
\end{tikzpicture}
  \end{tabular}
\caption{The throughput $\rho_5'$ is not concave, although $\rho_5$ and $\rho_5+\rho_5'$ still are}\label{fig:phasediagram2}
\end{figure}

The second new feature of this model compared with~\eqref{SAMU1} lies in the duplication of the physician's lane and the fact that very urgent calls (in proportion $\alpha$ among all calls transfered to doctors) are handled in priority. This has the effect of splitting each congestion phase associated with a lack of emergency physicians in two parts. Indeed,
given an MRA throughput $\rho_1$, define the two functions $\underline{N}_P$ and $\overline{N}_P$ by
\[
\underline{N}_P(\rho_1) \coloneqq  \pi\alpha(\tau_2+\tau_3)\rho_1 \qquad \text{and} \qquad \overline{N}_P(\rho_1) \coloneqq \pi(\tau_2+\tau_3)\rho_1 \,.
\]
A minimum number of $\overline{N}_P$ physicians is needed to handle all the calls passed by the MRAs \textit{via} the reservoir assistant. However, in case of a lack of physicians, the priority mechanism ensures that the very urgent calls remain handled as long as $N_P\geq\underline{N}_P$ (phases \phase{4\alpha}, \phase{5\alpha} and \phase{6\alpha}). Below the latter threshold, there are too few physicians to handle these very urgent calls (phases \phase{4}, \phase{5} and \phase{6}). Remark that in presence of priorities, the throughput function of transitions may not be concave anymore, see for instance $\rho_5'$ as a function of $N_A$ and $N_P$ in~\Cref{fig:phasediagram2} (supposing reservoir assistants are not limiting, thus $N_R\geq\overline{N}_R$). In this cross-section though, note that in addition to $\rho_1$ and $\rho_3$, both $\rho_5$ and $\rho_5+\rho_5'$ are still concave.

\medskip
A second qualitative advantage of the system~\eqref{eq:SAMU2} is that contrary to the model~\eqref{SAMU1}, we observe that a slowdown in the emergency physician circuit does not affect the throughput of the MRAs, as an effect of their desynchronization by the reservoir buffer. It may still happen that we encounter both a lack of MRAs and physicians (phases \phase{5\alpha} and \phase{5}), but the latter do not prevent the former to pick up inbound calls at their maximal possible throughput.

\medskip
It is instructive to study the situations in which the reservoir assistants are understaffed (phases \phase{3}, \phase{6\alpha} and \phase{6}), although such situations are not desirable in practice.

\medskip
In particular, we observe that throughputs $\rho_3$ and $\rho_5$ are proportional to $N_R/\tau_s - N_P/(\tau_s+\tau_3)$ in phase \phase{6\alpha} (it remains true for $\rho_3$ in phase \phase{6} as well), which means that increasing the number of emergency physicians slows down the handling of top priority calls! This establishes the following seemingly paradoxical property:\vspace*{-2mm}
\begin{paradox} In the presence of priority rulings, the asymptotic throughputs of some transitions (even one with highest level of priority) of the net may be decreased by an increase of the resources.\label{paradox:nonmonotonic}
\end{paradox}

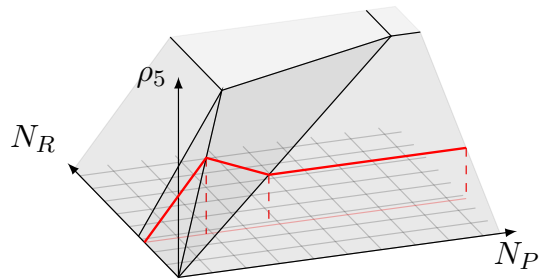
\begin{figure}[h]
  \hspace{-8mm}
  \centering
  \def\tkzscl{.48}  % ./phases_diagram_2_3D_NRNM.tex

\tdplotsetmaincoords{68}{-20}
\hskip-2ex
\scalebox{1.1}{
 \begin{tikzpicture}[tdplot_main_coords,scale=\tkzscl]

    \begin{scope}[scale=0.9]

  \draw[->,>=latex] (0,0,0) -- coordinate (x axis mid) (10,0,0) node[below] {$N_P$};
  \draw[->,>=latex] (0,0,0) -- coordinate (y axis mid) (0,9,0) node[above left] {$N_R$};
  \draw[->,>=latex] (0,0,0) -- (0,0,6) node[left] {$\rho_5$};

  \def\alphaa{0.33}
  \def\betaa{0.85}
  \def\xm{9.5}
  \def\ym{8.5}

    \def\lambdaa{1.8}
    \def\pii{0.3}
    \def\tauun{8}
    \def\tautrois{10}
    \def\taus{6}

    \def\scl{20}

  \foreach \x in {0,...,9}
    \draw[thin, opacity=.2] (\x,0)--(\x,\ym) ;
  \foreach \y in {0,...,8}
    \draw[thin, opacity=.2] (0,\y)--(\xm,\y) ;

  \coordinate (x0) at   (0,   0, 0     ) {};
  \coordinate (x1) at   (\xm, 0, 0     ) {};
  \coordinate (x2) at ($(0,   {\taus*\pii*\lambdaa}, 0)$) {};
  \coordinate (x3) at ($({\pii*\alphaa*\lambdaa*(\taus+\tautrois)},   {(1+\alphaa)*\taus*\pii*\lambdaa}, {\alphaa*\pii*\lambdaa*\scl})$) {};
  \coordinate (x4) at ($({\pii*\lambdaa*(\taus+\tautrois)},   {2*\taus*\pii*\lambdaa}, {\alphaa*\pii*\lambdaa*\scl})$) {};
  \coordinate (x5) at ($(\xm,   2*\taus*\pii*\lambdaa, {\alphaa*\pii*\lambdaa*\scl})$) {};
  \coordinate (x6) at ($(0,\ym, 0)$) {};
  \coordinate (x7) at ($({\pii*\alphaa*\lambdaa*(\taus+\tautrois)},\ym, {\alphaa*\pii*\lambdaa*\scl})$) {};
  \coordinate (x8) at ($({\pii*\lambdaa*(\taus+\tautrois)},\ym, {\alphaa*\pii*\lambdaa*\scl})$) {};
  \coordinate (x9) at ($(\xm,\ym, {\alphaa*\pii*\lambdaa*\scl})$) {};

  \coordinate (b0) at ($(0, {\betaa*\taus*\pii*\lambdaa}, 0)$) {};
  \coordinate (b1) at ($({\betaa*\alphaa*(\taus+\tautrois)*\pii*\lambdaa/(1+\alphaa)}, {\betaa*\taus*\pii*\lambdaa}, {\scl*\betaa*\alphaa*\pii*\lambdaa/(1+\alphaa)})$) {};
  \coordinate (b2) at ($({\betaa*(\taus+\tautrois)*\pii*\lambdaa/2}, {\betaa*\taus*\pii*\lambdaa}, {\scl*\betaa*\alphaa*\pii*\lambdaa/2})$) {};
  \coordinate (b3) at ($(\xm, {\betaa*\taus*\pii*\lambdaa}, {\scl*\betaa*\alphaa*\pii*\lambdaa/2})$) {};

  \def\zm{5}
  \coordinate (bt0) at ($(0, {\betaa*\taus*\pii*\lambdaa}, \zm)$) {};
  \coordinate (bt3) at ($(\xm, {\betaa*\taus*\pii*\lambdaa}, \zm)$) {};
  \coordinate (bb0) at ($(0, {\betaa*\taus*\pii*\lambdaa}, 0)$) {};
  \coordinate (bb1) at ($({\betaa*\alphaa*(\taus+\tautrois)*\pii*\lambdaa/(1+\alphaa)}, {\betaa*\taus*\pii*\lambdaa}, 0)$) {};
  \coordinate (bb2) at ($({\betaa*(\taus+\tautrois)*\pii*\lambdaa/2}, {\betaa*\taus*\pii*\lambdaa}, 0)$) {};
  \coordinate (bb3) at ($(\xm, {\betaa*\taus*\pii*\lambdaa},0)$) {};

\draw[red, thin, dashed] (b1) -- (bb1);
\draw[red, thin, dashed] (b2) -- (bb2);
\draw[red, thin, dashed] (b3) -- (bb3);
\draw[red, thin, opacity = .3] (bb0) -- (bb3);

  \filldraw[black, opacity=0.09] (x0) -- (x1) -- (x5) -- (x4) -- cycle;
  \filldraw[black, opacity=0.14] (x0) -- (x3) -- (x4) -- cycle;
  \filldraw[black, opacity=0.12] (x0) -- (x3) -- (x2) -- cycle;
  \filldraw[black, opacity=0.09] (x2) -- (x3) -- (x7) -- (x6) -- cycle;
  \filldraw[black, opacity=0.05] (x4) -- (x3) -- (x7) -- (x8) -- cycle;
  \filldraw[black, opacity=0.05] (x5) -- (x4) -- (x8) -- (x9) -- cycle;

\draw (x0) -- (x3) -- (x2);
\draw (x7) -- (x3) -- (x4) -- (x8);
\draw (x5) -- (x4) -- (x0);

  \draw[red, thick] (b0) -- (b1) -- (b2) -- (b3);
\end{scope}
\end{tikzpicture} } \vspace*{2mm}
\caption{The throughput $\rho_5$ does not decrease with respect to $N_P$}\label{fig:nonmonotonicthroughput}
\vspace{-1mm}
\end{figure}

We depict this phenomenon on~\Cref{fig:nonmonotonicthroughput} (supposing that first-answering MRAs are non-limiting so $N_A\geq N_A^*$), the red curve showing that $\rho_5$ is nonmonotic as $N_P$ grows and as we go through phases \phase{6}, then \phase{6\alpha} and finally \phase{3}. This counter-intuitive situation can be explained as follows: suppose for sake of simplicity that $N_R < \pi\lambda\tau_s$, so that there are not enough reservoir assistants to even fill the reservoir room (while twice this amount of agents would be needed to fill it and empty it). At $N_P = 0$, both very urgent and urgent calls queues build up in the reservoir at throughput $N_R/\pi\tau_s$ (maximum filling speed of reservoir agents). As $N_P$ increases (phase \phase{6}), some very urgent cases can now be handled by emergency physicians at rate $N_P/{(\tau_s+\tau_3)}$, however this task requires a second accompaniment step with reservoir assistants and is prioritized to them. Hence, they spend less time filling the reservoir and $\rho_3$ decreases. As $N_P$ increases again (phase \phase{6\alpha}), there are enough emergency physicians to also pick-up calls from the second-priority room, requiring again the intervention of reservoir assistants (before admitting new patients), as a result $\rho_3$ decreases again and so does $\rho_5=\alpha\rho_3$: the reservoir assistants have ``less time'' to admit and detect very urgent calls as they must escort already admitted very urgent and urgent calls to doctors before, and eventually less very urgent calls are handled. We insist again that this unusual phenomenon arises because top-priority transitions are served downstream of some inferior-priority ones. This echoes a similar pathological behavior observed in~\cite{ye2007paradox}.

\definecolor{gray3}{gray}{0.91}
\definecolor{gray4}{gray}{0.95}
\setlength{\arrayrulewidth}{1.3pt}

  \def\tblscl{.188}
  \def\alphaa{0.3}
  \def\showtag{1}
  \begin{figure}[h]
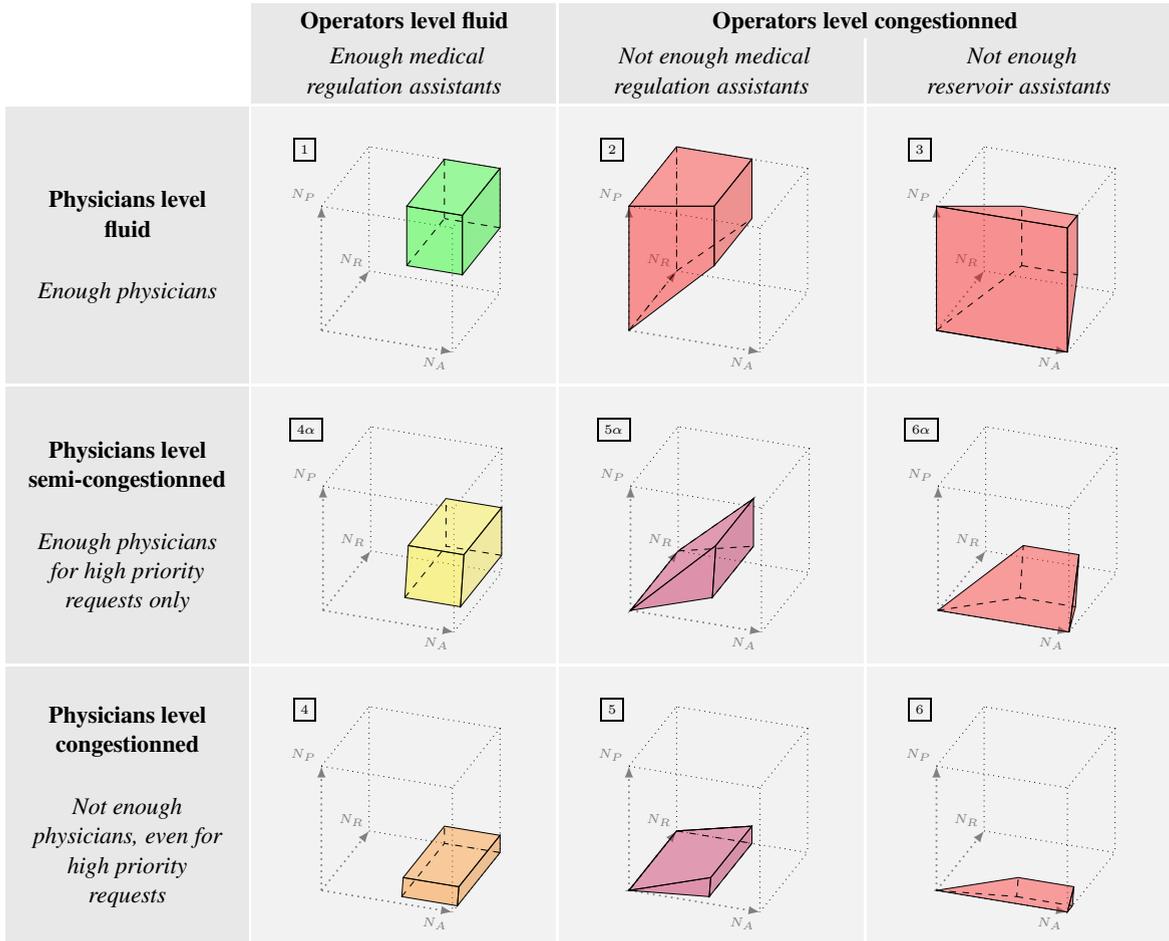

  \vspace*{3mm}
      \resizebox{\textwidth}{!}{
\begin{tabular}{!{\color{white}{\vrule width 1.3pt}}C{3.3cm}!{\color{white}{\vrule width 1.3pt}}C{4.3cm}!{\color{white}{\vrule width 1.3pt}}C{4.3cm}!{\color{white}{\vrule width 1.3pt}}C{4.3cm}!{\color{white}{\vrule width 1.3pt}}}
        \arrayrulecolor{white}
        \hline
      &
      \cellcolor{gray3}\textbf{Operators level fluid} &
       \multicolumn{2}{c|}{\cellcolor{gray3}\textbf{Operators level congestionned}}\\

      \cline{3-4}
      & \cellcolor{gray3} \textit{Enough medical \newline regulation assistants} &
      \cellcolor{gray3} \textit{Not enough medical \newline regulation assistants} &
      \cellcolor{gray3} \textit{Not enough \newline reservoir assistants} \\
\hline
      \end{tabular}
  }\vspace*{2mm}
  \caption{Congestion phases of the model~\eqref{eq:SAMU2}}
  \label{fig:phasesRes}%\vspace*{-2mm}
  \end{figure}

  \section{Concluding remarks}

  We developed a model of fluid timed Petri net including both preselection
  and priority routings. In the absence of priority, we showed
  that the dynamics of the net is equivalent to the Bellman equation
  of a semi-Markov decision problem, from which a number
  of properties follows: existence and universality
  of the throughput vector (independence from the initial condition),
  existence of stationary regimes by reduction to a lexicographic system,
  polynomial-time computability of the throughput by reduction to a linear program, and explicit representation of the throughput, as a function
  of resources, by a polyhedral complex. This approach provides
  tools to address further issues: e.g., an important practical
  problem is to bound the time needed to absorb a peak of congestion.
  We believe it can still be addressed using techniques
  of nonexpansive dynamical systems, along lines of
  ~\cite{Schweitzer1978a,spectral}, we
  leave this for a subsequent work.

  In the presence of
  priority, only part of these results remain: finding stationary
  regimes is equivalent to solving a lexicographic system,
  which is a system of polynomial equations over a tropical
  semifield of germs. In other words, stationary regimes
  are the points of a tropical variety, and we still get a polyhedral
  complex, describing all the congestion phases.
    This complex can be computed
  in exponential time, by enumerating strategies, as we did on our case study. Whereas we do not
  expect worst-case polynomial-time computability results in such a generality
  (solving tropical polynomial systems is generally NP-hard), we leave it
  for further work to get finer complexity bounds. It is also an open
  problem to compare the asymptotic behavior of counters, for an arbitrary
  initial condition, with stationary solutions.

\subsection*{Acknowledgment}

This work was done through a collaboration with the SAMU of AP-HP. We wish to thank especially, \textsc{Pr}.~P.~Carli, \dr~\'E.~Chanzy, \dr~\'E.~Lecarpentier, \dr~Ch.~Leroy, \dr~Th.~Loeb, \dr~J.-S.~Marx, \dr~N.~Poirot and \dr~C.~Telion for making this work possible, for their support and for insightful comments. We also thank all the other personals of the SAMU, in particular \dr~J.~Boutet, J.-M.~Gourgues, I.~Lhomme, F.~Linval and Th.~P\'erennou.
The present work was developed as part of an ongoing
study within a project
between AP-HP and Préfecture de Police (PP)
aiming at an interoperability among the
different call centers. It strongly benefited from
the experience acquired, since 2014, on the analysis of the new platform ``PFAU'' (answering to the emergency numbers 17, 18 and 112), developed
by PP.  We thank \textsc{Lcl}~S.~Raclot and R.~Reboul,
in charge of the PFAU project at PP, for their constant support and insightful
comments all along these years. This work also strongly benefited of scientific discussions with Ph.~Robert, whom we thank. Finally, we thank the reviewers of the Petri nets 2020 conference in which a first version of this work was accepted for their careful and detailed reviews, as well as the organizing comittee, who have made this enhanced version possible. We also thank the reviewers of the present article for their very careful and detailed comments.

\bibliographystyle{fundam}

\begin{thebibliography}{}
\providecommand{\url}[1]{\texttt{#1}}
\providecommand{\urlprefix}{URL }
\expandafter\ifx\csname urlstyle\endcsname\relax
  \providecommand{\doi}[1]{doi:\discretionary{}{}{}#1}\else
  \providecommand{\doi}{doi:\discretionary{}{}{}\begingroup
  \urlstyle{rm}\Url}\fi
\providecommand{\eprint}[2][]{\url{#2}}

\end{thebibliography}


\begin{thebibliography}{10}
\providecommand{\url}[1]{\texttt{#1}}
\providecommand{\urlprefix}{URL }
\expandafter\ifx\csname urlstyle\endcsname\relax
  \providecommand{\doi}[1]{doi:\discretionary{}{}{}#1}\else
  \providecommand{\doi}{doi:\discretionary{}{}{}\begingroup
  \urlstyle{rm}\Url}\fi
\providecommand{\eprint}[2][]{\url{#2}}

\bibitem{bcoq}
Baccelli F, Cohen G, Olsder GJ, Quadrat JP.
 Synchronization and Linearity.
 Wiley, 1992. ISBN: 047193609X.

\bibitem{how06}
Heidergott G, Olsder GJ, van~der Woude J.
\newblock Max Plus at work.
 Princeton University Press, 2006. ISBN:9780691117638.

\bibitem{KOMENDA}
Komenda J, Lahaye S, Boimond JL, van~den Boom T.
\newblock Max-plus algebra in the history of discrete event systems.
\newblock \emph{Annual Reviews in Control}, 2018.
\newblock \textbf{45}:240--249.  doi:10.1016/j.arcontrol.2018.04.004.

\bibitem{CGQ95b}
Cohen G, Gaubert S, Quadrat J.
\newblock Asymptotic Throughput of Continuous Timed {P}etri Nets.
\newblock In: Proceedings of the 34th Conference on Decision and Control. New
  Orleans, 1995.  doi:10.1109/CDC.1995.480646.

\bibitem{CGQ95a}
Cohen G, Gaubert S, Quadrat J.
\newblock Algebraic System Analysis of Timed {P}etri Nets.
\newblock In: Gunawardena J (ed.), Idempotency, Publications of the Isaac
  Newton Institute, pp. 145--170. Cambridge University Press, 1998.

\bibitem{gaujal2004optimal}
Gaujal B, Giua A.
\newblock Optimal stationary behavior for a class of timed continuous {P}etri  nets.
  \newblock \emph{Automatica}, 2004.
\newblock \textbf{40}(9):1505--1516.  doi:10.1016/j.automatica.2004.04.018.

\bibitem{recalde}
Recalde L, Silva M.
\newblock Petri net fluidification revisited: Semantics and steady state.
\newblock \emph{European Journal of Automation, APII-JESA}, 2001.
\newblock \textbf{35}(4):435--449.

\bibitem{allamigeon2015performance}
Allamigeon X, B{\oe}uf V, Gaubert S.
\newblock Performance evaluation of an emergency call center: tropical
  polynomial systems applied to timed {P}etri nets.
\newblock In: International Conference on Formal Modeling and Analysis of Timed
  Systems. Springer, 2015 pp. 10--26.  doi:10.1007/978-3-319-22975-1\_2.

\bibitem{lecuyer}
L'Ecuyer P, Gustavsson K, Olsson L.
\newblock Modeling bursts in the arrival process to an emergency call center.
\newblock In: Rabe M, Juan AA, Mustafee N, Skoogh A, Jain S, Johansson B
  (eds.), Proceedings of the 2018 Winter Simulation Conference. 2018.
  doi:10.1109/WSC.2018.8632536.

\bibitem{boeufrobert}
Boeuf V, Robert P.
\newblock {A Stochastic Analysis of a Network with Two Levels of Service}.
\newblock \emph{{Queueing Systems}}, 2019.
\newblock \textbf{92}(3-4):30.  doi:10.1007/s11134-019-09617-y.

\bibitem{puterman2014markov}
Puterman ML.
\newblock Markov decision processes: discrete stochastic dynamic programming.
\newblock John Wiley \& Sons, 2014. ISBN:978-1-118-62587-3.

\bibitem{yushkevich1982semi}
Yushkevich A.
\newblock On semi-{M}arkov controlled models with an average reward criterion.
\newblock \emph{Theory of Probability \& Its Applications}, 1982.
\newblock \textbf{26}(4):796--803.  doi:10.1137/1126085.

\bibitem{jianyong2004average}
Jianyong L, Xiaobo Z.
\newblock On average reward semi-Markov decision processes with a general
  multichain structure.
\newblock \emph{Mathematics of Operations Research}, 2004.
\newblock \textbf{29}(2):339--352.  doi:10.1287/moor.1030.0077.

\bibitem{bermanandplemmons}
Berman A, Plemmons RJ.
\newblock Nonnegative matrices in the mathematical sciences.
\newblock Academic Press, 1979.

\bibitem{SchweitzerFedergruenEquation}
Schweitzer PJ, Federgruen A.
\newblock The Functional Equations of Undiscounted {M}arkov Renewal
  Programming.
\newblock \emph{Mathematics of Operations Research}, 1978.
\newblock \textbf{3}(4):308--321. doi:10.1287/moor.3.4.308.

\bibitem{Denardo-Fox68}
Denardo E, Fox B.
\newblock Multichain {M}arkov {R}enewal {P}rograms.
\newblock \emph{SIAM J.Appl.Math}, 1968.
\newblock \textbf{16}:468--487.

\bibitem{crandall}
Crandall M, Tartar L.
\newblock Some relations between non expansive and order preserving maps.
\newblock \emph{Proceedings of the AMS}, 1980.
\newblock \textbf{78}(3):385--390.  doi:10.2307/2042330.

\bibitem{agn12}
Akian M, Gaubert S, Nussbaum R.
\newblock Uniqueness of the fixed point of nonexpansive semidifferentiable
  maps.
\newblock \emph{Trans. of AMS}, 2016.
\newblock \textbf{368}(2):1271--1320.
\newblock \eprint{1201.1536}.

\bibitem{kohlberg80}
Kohlberg E.
\newblock Invariant half-lines of nonexpansive piecewise-linear
  transformations.
\newblock \emph{Math. Oper. Res.}, 1980.
\newblock \textbf{5}(3):366--372.

\bibitem{martus}
Martus P.
\newblock Asymptotic properties of nonstationary operator sequences in the
  nonlinear case.
\newblock Phd thesis, Friedrich-Alexander Universit\"at, Erlangen-N\"urnberg,   1989.

\bibitem{LN12}
Lemmens B, Nussbaum RD.
\newblock Nonlinear {P}erron-{F}robenius theory, volume 189 of \emph{Cambridge
  Tracts in Mathematics}.
\newblock Cambridge University Press, Cambridge, 2012.
 ISBN:978-0-521-89881-2.

\bibitem{LS2}
Lemmens B, Scheutzow M.
\newblock On the dynamics of sup-norm non-expansive maps.
\newblock \emph{Ergodic Theory Dynam. Systems}, 2005.
\newblock \textbf{25}(3):861--871. doi:10.1017/S0143385704000665.

\bibitem{spectral}
Akian M, Gaubert S.
\newblock Spectral Theorem for Convex Monotone Homogeneous Maps, and ergodic
  Control.
\newblock \emph{Nonlinear Analysis. Theory, Methods \& Applications}, 2003.
\newblock \textbf{52}(2):637--679.
\newblock \eprint{math.SP/0110108}.

\bibitem{GauGu}
Gaubert S, Gunawardena J.
\newblock The {P}erron-{F}robenius Theorem for Homogeneous, Monotone Functions.
\newblock \emph{Trans. of AMS}, 2004.
\newblock \textbf{356}(12):4931--4950.
\newblock \eprint{math.FA/0105091}.

\bibitem{schweitzer79}
Schweitzer P, Federgruen A.
\newblock Geometric convergence of value-iteration in multichain {M}arkov
  decision problems.
\newblock \emph{Adv. Appl. Prob.}, 1979.
\newblock \textbf{11}:188--217.  doi:10.2307/1426774.

\bibitem{de2010triangulations}
De~Loera JA, Rambau J, Santos F.
\newblock Triangulations Structures for algorithms and applications.
\newblock Springer, 2010.  doi:10.1007/978-3-642-12971-1.

\bibitem{ye2007paradox}
Ye HQ.
\newblock A paradox for admission control of multiclass queueing network with
  differentiated service.
\newblock \emph{Journal of applied probability}, 2007.
\newblock \textbf{44}(2):321--331.  doi:10.1239/jap/1183667404.

\bibitem{Schweitzer1978a}
Schweitzer P, Federgruen A.
\newblock The asymptotic behavior of undiscounted value iteration in {M}arkov
  decision problems.
\newblock \emph{Math. of O.R.}, 1978.
\newblock \textbf{2}:360--381. doi:10.1287/moor.2.4.360.
\end{thebibliography}

          \clearpage

\appendix

\section{Phases Equations for Model~\eqref{eq:SAMU2}}
\label{ap:congestionphasesdetailed}

\begin{center}
\resizebox{!}{.455\textheight}{
\makebox[\textwidth][c]{
  \setlength{\arrayrulewidth}{.1pt}
\begin{tabular}{R{.6cm}C{9cm}|C{3cm}|C{2cm}|C{3cm}}
  \multicolumn{2}{l|}{Phase \hskip17ex Bounding inequalities} & $\rho_1$ & $\rho_5$ & $\rho_5'$ \\
  &  & & &\\[-2.5ex] \hline & & &\\[-1.5ex]
  $\boxed{1}$ & $\displaystyle\frac{N_A}{\tau_1+\pi\tau_s} \geq \lambda \qquad ; \qquad \displaystyle\frac{N_R}{2\tau_s}\geq \pi\lambda$ \newline[1.8ex] $\displaystyle\frac{N_P}{\tau_s+\tau_3}\geq \pi\lambda$ & $\lambda$ & $\pi\alpha\lambda$ & $\pi(1-\alpha)\lambda$\\
  &  & & &\\[-1.5ex]\hline & & &\\[-1.5ex]
  $\boxed{4\alpha}$ & $\displaystyle\frac{N_A}{\tau_1+\pi\tau_s} \geq \lambda \qquad ; \qquad \displaystyle\frac{N_R}{\tau_s}\geq \pi\lambda+\frac{N_P}{\tau_s+\tau_3}$ \newline[1.8ex] $\pi\alpha\lambda \leq \displaystyle\frac{N_P}{\tau_s+\tau_3}\leq \pi\lambda$ & $\lambda$ & $\pi\alpha\lambda$ & $\displaystyle\frac{N_P}{\tau_s+\tau_3}-\pi\alpha\lambda$\\
  &  & & &\\[-1.5ex]\hline & & &\\[-1.5ex]
  $\boxed{4}$ & $\displaystyle\frac{N_A}{\tau_1+\pi\tau_s} \geq \lambda \qquad ; \qquad \displaystyle\frac{N_R}{\tau_s}\geq \pi\lambda+\frac{N_P}{\tau_s+\tau_3}$ \newline[1.8ex] $\displaystyle\frac{N_P}{\tau_s+\tau_3} \leq \pi\alpha\lambda$ & $\lambda$ & $\displaystyle\frac{N_P}{\tau_s+\tau_3}$ & 0\\

    &   & & &\\[-1.5ex]\hline & & &\\[-1.5ex]
    $\boxed{2}$ & $\displaystyle\frac{N_A}{\tau_1+\pi\tau_s} < \lambda \qquad ; \qquad \displaystyle\frac{N_R}{2\tau_s} \geq \frac{\pi N_A}{\tau_1+\pi\tau_s}$ \newline[1.8ex] $\displaystyle\frac{N_P}{\tau_s+\tau_3} \geq \frac{\pi N_A}{\tau_1+\pi\tau_s}$ & $\displaystyle\frac{N_A}{\tau_1+\pi\tau_s}$ & $\displaystyle\frac{\pi\alpha N_A}{\tau_1+\pi\tau_s}$ & $\displaystyle\frac{\pi(1-\alpha)N_A}{\tau_1+\pi\tau_s}$\\
    &   & & &\\[-1.5ex]\hline & & &\\[-1.5ex]
    $\boxed{5\alpha}$ & $\displaystyle\frac{N_A}{\tau_1+\pi\tau_s} < \lambda \qquad ; \qquad \frac{N_R}{\tau_s} \geq \displaystyle\frac{\pi N_A}{\tau_1+\pi\tau_s} + \frac{N_P}{\tau_3+\tau_s}$ \newline[1.8ex] $\displaystyle \frac{\pi\alpha N_A}{\tau_1+\pi\tau_s} \leq \frac{N_P}{\tau_s+\tau_3} \leq \displaystyle\frac{\pi N_A}{\tau_1+\pi\tau_s} $  & $\displaystyle\frac{N_A}{\tau_1+\pi\tau_s}$ & $\displaystyle\frac{\pi\alpha N_A}{\tau_1+\pi\tau_s}$ & $\displaystyle\frac{N_P}{\tau_s+\tau_3}-\frac{\pi\alpha N_A}{\tau_1+\pi\tau_s}$\\
    &   & & &\\[-1.5ex]\hline & & &\\[-1.5ex]
    $\boxed{5}$ & $\displaystyle\frac{N_A}{\tau_1+\pi\tau_s} < \lambda \qquad ; \qquad \frac{N_R}{\tau_s} \geq \displaystyle\frac{\pi N_A}{\tau_1+\pi\tau_s} + \frac{N_P}{\tau_3+\tau_s}$ \newline[1.8ex] $\displaystyle\frac{N_P}{\tau_3+\tau_s} \leq \displaystyle\frac{\pi\alpha N_A}{\tau_1+\pi\tau_s}$ & $\displaystyle\frac{N_A}{\tau_1+\pi\tau_s}$ & $\displaystyle\frac{N_P}{\tau_s+\tau_3}$ & 0\\

    &   & & &\\[-1.5ex]\hline & & &\\[-1.5ex]
    $\boxed{3}$ & $\displaystyle\frac{N_R}{2\tau_s}<\pi\lambda  \qquad ; \qquad \displaystyle\frac{N_R}{2\tau_s} \leq \frac{\pi N_A}{\tau_1+\pi\tau_s}$ \newline[1.8ex] $\displaystyle\frac{N_P}{\tau_s+\tau_3} \geq \frac{N_R}{2\tau_s}$ & $\displaystyle\frac{N_R}{2\pi\tau_s}$ & $\displaystyle\frac{\alpha N_R}{2\tau_s}$ & $\displaystyle\frac{(1-\alpha)N_R}{2\tau_s}$\\
    &   & & &\\[-1.5ex]\hline & & &\\[-1.5ex]
    $\boxed{6\alpha}$ & $\displaystyle\frac{N_R}{\tau_s} - \frac{N_P}{\tau_s+\tau_3} < \pi \lambda \qquad ; \qquad \displaystyle\frac{\pi N_A}{\tau_1+\pi\tau_s} \geq \frac{N_R}{\tau_s} - \frac{N_P}{\tau_s+\tau_3}$ \newline[1.8ex]  $\displaystyle\frac{\alpha N_R}{(1+\alpha)\tau_s} \leq \frac{N_P}{\tau_s+\tau_3} \leq \frac{N_R}{2\tau_s}$ & $\displaystyle\frac{N_R}{\pi\tau_s}-\frac{N_P}{\pi(\tau_s+\tau_3)}$ & $\displaystyle\frac{\alpha N_R}{\tau_s}-\frac{\alpha N_P}{\tau_s+\tau_3}$ & $\displaystyle \frac{(1+\alpha)N_P}{\tau_s+\tau_3}-\frac{\alpha N_R}{\tau_s}$\\
    &   & & &\\[-1.5ex]\hline & & &\\[-1.5ex]
    $\boxed{6}$ & $\displaystyle \frac{N_R}{\tau_s}-\frac{N_P}{\tau_s+\tau_3}<\pi\lambda  \qquad ; \qquad \displaystyle\frac{\pi N_A}{\tau_1+\pi\tau_s}\geq \displaystyle \frac{N_R}{\tau_s}-\frac{N_P}{\tau_s+\tau_3} $ \newline[1.8ex] $\displaystyle\frac{N_P}{\tau_s+\tau_3} \leq \frac{\alpha N_R}{(1+\alpha)\tau_s}$ & $\displaystyle\frac{N_R}{\pi\tau_s}-\frac{N_P}{\pi(\tau_s+\tau_3)}$ & $\displaystyle\frac{N_P}{\tau_s+\tau_3}$ & 0
    \end{tabular}
    }
    }
  \end{center}

\end{document}